\documentclass{siamart220329bis}
\usepackage{ulem}

\usepackage{algorithmic}
\ifpdf
  \DeclareGraphicsExtensions{.eps,.pdf,.png,.jpg}
\else
  \DeclareGraphicsExtensions{.eps}
\fi

\usepackage{amsopn}

\usepackage[frenchb,english]{babel}
\usepackage[utf8]{inputenc}
\usepackage[T1]{fontenc}
\usepackage{lipsum}
\usepackage{fullpage}
\usepackage{amsmath}
\usepackage{amstext}
\usepackage{amsfonts}

\usepackage{amssymb}
\usepackage{amsfonts}
\usepackage{bbold}
\usepackage{cleveref}
\usepackage[ntheorem]{empheq}
\usepackage{xcolor}
\usepackage{tikz}
\usetikzlibrary{patterns}
\usepackage{ stmaryrd }
\usepackage{accents}

\newcommand{\ubar}[1]{\underaccent{\bar}{#1}}
\newcommand{\essup}{\ess \!\!\; \sup}
\newcommand{\essinf}{\ess \!\!\; \inf}

\DeclareMathOperator\supp{supp}
\DeclareMathOperator{\sign}{sign}

\DeclareMathOperator{\meas}{meas}

\DeclareMathOperator{\ess}{ess}
\DeclareMathOperator{\God}{\bold{God}}
\DeclareMathOperator{\card}{\bold{Card}}

\renewcommand{\d}{\textrm{\,d}}

\newsiamremark{remark}{Remark}
\newsiamremark{hypothesis}{Hypothesis}
\crefname{hypothesis}{Hypothesis}{Hypotheses}
\newsiamthm{claim}{Claim}

\title{Existence of solutions to a class of one-dimensional models for pedestrian evacuations\thanks{Submitted to the editors DATE.}}

\author{Boris Andreianov\thanks{Institut Denis Poisson CNRS UMR 7013, Université de Tours, Université d’Orléans,
Parc Grandmont, 37200 Tours, France and
Peoples’ Friendship University of Russia (RUDN University)
6 Miklukho-Maklaya St, Moscow, 117198, Russian Federation
(\email{Boris.Andreianov@lmpt.univ-tours.fr}, https://www.idpoisson.fr/andreianov/).}
\and Theo Girard \thanks{Institut Denis Poisson, Université de Tours, Parc Grandmont, 37200 Tours, France (\email{theo.girard@lmpt.univ-tours.fr}).}
}

\headers{An existence result for Hughes' model}{B. ANDREIANOV, T. GIRARD}

\ifpdf
\hypersetup{
  pdftitle={An existence result for Hughes' model},
  pdfauthor={ B. Andreianov, T. Girard}
}
\fi

\begin{document}

\maketitle

\begin{abstract}
In the framework inspired by R. L. Hughes model (Transp. Res. B, 2002) for pedestrian evacuation in a corridor, we establish existence of a solution by a topological fixed point argument. This argument applies to a class of models where the dynamics of the pedestrian density $\rho$ (governed by a discontinuous-flux Lighthill,Whitham and Richards model $\rho_t + \left(\sign(x-\xi(t))\rho v(\rho)\right)_x = 0$ ) is coupled via an abstract operator to the computation of a Lipschitz continuous ``turning curve'' $\xi$. \textcolor{black}{We illustrate this construction by several examples, including the standard Hughes' model with affine cost, and either with open-end conditions or with conditions corresponding to panic behaviour with capacity drop at exits. Other examples put forward versions of the Hughes model with inertial dynamics of the turning curve and general costs.} \end{abstract}



\section{Introduction}

\subsection{The Hughes model and its variants}
The Lighthill,Whitham and Richards (LWR) model for traffic introduced in \cite{lighthill1955kinematic} and in \cite{richards1956shock} consists in a conservation law for the vehicule density $\rho$ with a concave positive flux $\rho v(\rho)$:
\begin{equation}\label{eqLWR}
  \left\lbrace
  \begin{matrix}
  \rho_t + \left[\rho v(\rho) \right]_x &=& 0 \\
  \rho(t=0,x) &=& \rho_0(x). \\
\end{matrix}
  \right.
\end{equation}
Here, we can suppose that the density $\rho$ takes its values in $[0,1]$ and $v$ stands for the speed of the traffic. This model can be seen as the mass conservation equation where velocity $v$ depends only on the traffic density $\rho$. One frequently chooses $v(\rho) = 1 - \rho$ up to a multiplicative constant representing the maximal velocity. This describes a transport of the initial density of agents $\rho_0$ at $t=0$  towards $x = + \infty$ where the speed is decreasing when the density of agents is increasing.

\smallskip
Then, in \cite{Hughes2002ACT}, Hughes proposed a model of pedestrian evacuation as a system of two equations on $\rho$ and $\phi$ which is known as Hughes' model. In the multi-dimensional model, $\rho$ is the density of pedestrians with respect to time $t$ and space $x$. The dynamics of $\rho$ is governed by LWR conservation laws with direction field oriented towards the exits of a bounded domain $\Omega$.
In order to prescribe the direction towards the exit preferred by a pedestrian at location $x$ at a time $t$, Hughes defines $\phi(t,x)$, the ``potential field'' satisfying an eikonal equation. The potential $\phi$ is zero on the exits located on $\partial \Omega$. A pedestrian would then choose to ``descend the gradient'' of this potential in order to leave the domain $\Omega$ by these exits. Theory of the Hughes' model is yet incomplete, even in one space dimension. In the 1D case, the model of \cite{Hughes2002ACT} takes the form:
\begin{subequations}\label{eqHughesModel}
\begin{empheq}[left = \empheqlbrace]{align}
\rho_t + \left[ \sign(-\partial_x \phi)\rho v(\rho)\right]_x &= 0 \label{seqCLHughes} \\
\rho(t, x = \pm 1) & = 0 \label{seqEikoCB-rho}\\
\left|\partial_x \phi \right| &= \frac{1}{v(\rho)}  \label{seqEiko} \\
\phi(t, x = \pm 1) &= 0 . \label{seqEikoCB}
\end{empheq}
\end{subequations}
This problem \eqref{eqHughesModel} is set up in a corridor with two exits; upon renormalization, we assumed that $\Omega = \, (-1,1)$ and that the exits are located at $x = \pm 1$. At $t=0$ the pedestrians are distributed with a given density $\rho_0$ defined in $[-1,1]$ and at $t>0$, the pedestrians want to leave the corridor by either one of the exits (as if a fire alarm starts ringing at $t=0$).
The pedestrians move forward (with the positive flux $\rho \mapsto +\rho v(\rho)$) or backward (with $\rho \mapsto - \rho v(\rho)$ ) depending of the sign of $\partial_x \phi$. This results in \eqref{seqCLHughes} being a discontinuous flux LWR conservation law.
The sign of $\partial_x \phi$ is prescribed by the eikonal equation \eqref{seqEiko} where $c(\rho) = \frac{1}{v(\rho)}$ is a cost function that is high where the crowd is slow. Consequently, the pedestrians tend to avoid those ``congested'' regions.

\smallskip
The Dirichlet boundary condition \eqref{seqEikoCB-rho} on the density $\rho$ is understood in the Bardos-LeRoux-N\'ed\'elec sense standard for scalar conservation laws; it is shown in \cite[Sect.\,3]{andreianov2021} that upon extending $\rho_0$ by the value zero on $\mathbb{R}\setminus{[-1,1]}$,
one can replace the initial-boundary value problem \eqref{seqCLHughes}-\eqref{seqEikoCB-rho} with $\rho_0: (-1,1)\longrightarrow [0,1]$ by the pure intitial-value problem for \eqref{seqCLHughes} with the extended datum $\rho_0: \mathbb{R}\longrightarrow [0,1]$
(the extension means that $\rho_0$, now defined on $\mathbb{R}$, is supported in $[-1,1]$). We adopt this viewpoint and require, throughout the paper,
\begin{equation}\label{eq:HypRho_0}
	\rho_0 \in L^\infty(\mathbb{R};[0,1]), \quad \rho(x)=0\; \text{for $x\notin [-1,1]$};
\end{equation}
note that being compactly supported, $\rho_0\in L^1(\mathbb{R})$.
Assumption \eqref{eq:HypRho_0} for the conservation law \eqref{seqCLHughes} set up in the whole space can be seen as ``open-end condition'' at exits; we refer to Section~\ref{sec:constrained-exit} for models with more involved exit behavior.

\smallskip
In \cite{ElKhatib2013OnEW}, the 1D Hughes' model \eqref{eqHughesModel} has been reformulated in terms of a ``turning curve'' $\xi(t)$ instead of the potential $\phi$. Following the turning curve approach, our prototype model in the sequel will be:
\begin{subequations}\label{eqElKhatibModel}
\begin{empheq}[left = \empheqlbrace ]{align}
\rho_t + \left[ \sign(x-\xi(t))\rho v(\rho)\right]_x  =  0 \label{seqCLElKhatib}\\
\int_{-1}^{\xi(t)} c(\rho(t,x)) \d x  = \int_{\xi(t)}^1 c(\rho(t,x)) \d x .  \label{seqElKhatibEiko}
\end{empheq}
\end{subequations}
with $\rho$ defined for $t\in [0,T]$, $T>0$, and $x\in\mathbb{R}$ and with initial datum of the form \eqref{eq:HypRho_0}.
Here $c$ denotes a generic cost function. It is proven in \cite{ElKhatib2013OnEW} that we can equivalently consider either the Hughes' model potential equation \eqref{seqEiko}-\eqref{seqEikoCB} or the reformulated problem \eqref{seqElKhatibEiko} with the cost function $c(\rho) = \frac{1}{v(\rho)}$. \\
However, here, we will consider a cost verifying the following conditions:
\begin{equation}\label{eqCostConditions}
\left\lbrace \begin{matrix}
c \in W^{1,\infty}([0,1]),\\
\forall \rho \in [0,1], c(\rho) \geq 1, \\
c \textrm{ is increasing on } [0,1].
\end{matrix} \right.
\end{equation}

In \eqref{eqElKhatibModel}, $\rho$ is considered to be an entropy solution to \eqref{seqCLElKhatib}. Such notion of solution with a particular attention to the admissibility of the jump of $\rho$ across the turning curve $x = \xi(t)$ was proposed in \cite{ElKhatib2013OnEW} (we will slightly simplify this solution notion).
On the other hand, $\xi$ is a pointwise defined solution to \eqref{seqElKhatibEiko} whose existence in $L^\infty$ and uniqueness follows from the intermediate values theorem under the conditions \eqref{eqCostConditions}.\\
In this paper, we will consider a class of ``turning curve'' model's generalisations, keeping in mind the fact that, even in the setting \eqref{eqElKhatibModel}, little is known about the well-posedness of the Hughes' model. For notation's sake, we consider $f$ a generic concave positive flux such that $f(0)=f(1)=0$ (one can assume $f(\rho) = \rho v (\rho)$ to recover the LWR model):
\begin{subequations}\label{eqConsideredModel}
\begin{empheq}[left = \empheqlbrace, right = . ]{align}
\rho_t + \left[\sign(x-\xi(t))f(\rho)\right]_x &= 0 \label{seqTrueCL} \\
\rho(0,x) &= \rho_0(x) \label{seqTrueCI} \\
\xi &=\mathcal{I}(\rho) \label{seqTrueOp}
\end{empheq}
\end{subequations}
Here $\mathcal{I}$ is an abstract operator mapping the density $\rho$ to a turning curve $\xi$. The problem \eqref{eqElKhatibModel} is a particular case of \eqref{eqConsideredModel} where $\mathcal{I}$ is the solver of the integral equation \eqref{seqElKhatibEiko}.
Stating \eqref{seqTrueCI}, we mean that $\rho_0$ fulfills \eqref{eq:HypRho_0} which corresponds to open-end evacuation at exits, as stated above.

\smallskip
Let us briefly discuss known results on the specific problem \eqref{eqElKhatibModel} and its variants.
In \cite{ElKhatib2013OnEW} uniqueness is proven for a definition of entropy solutions taking the discontinuity into account but considering $\xi$ as being given beforehand (we will revisit this result in Section 2).
In \cite{Amadori2014-uw} global existence for Hughes' model (with $c(\rho) = \frac{1}{v(\rho)}$) is proven if one assumes that the density at the turning curve is zero for all times. In \cite{andreianov2021}, a uniqueness result in the same setting as this paper assuming moreover the $BV$ regularity of the solutions is provided.
And in \cite{Twarogowska2013}, \cite{Goatin2013-ha} and \cite{Gomes2016-wy} one can find numerical studies of the model.
Proof of existence and unicity for the regularized problem can be found in \cite{Di_Francesco2011-ke}.
The Hughes' model is also revisited with different turning curve equation in \cite{Carillo2016} with numerical simulation. In this paper, the authors introduce a regularization by convolution of the density named the subjective density. We also use the same type of idea when applying our main result in the case of a general cost function $c$.
The only general (with respect to the choice of the initial data) existence result is contained in \cite{andreianov2021}, where solutions with $BV_{loc}$ regularity away from the turning curve were constructed via a well-chosen many-particle approximation.
The result of \cite{andreianov2021} for problem \eqref{eqElKhatibModel} is limited to the case of an affine cost $c(\rho) = 1 + \alpha \rho$. Our result for the original setting \eqref{eqElKhatibModel} will also be limited to the affine cost case.
But we provide a shorter and less specific argument, compared to the many-particle approximation of \cite{andreianov2021}, also we require fewer assumptions on the velocity profile $v$ compared to \cite{andreianov2021}. The fixed-point approach we develop appears to be rather flexible since it  permits to handle several models of the form \eqref{eqConsideredModel}. We also adapt the arguments to more realistic, in the setting of crowd evacuation, exit behavior of the ``capacity drop'' kind (cf. \cite{Andreianov2014CDA,Donadello2015}).
However, we highlight the fact that our approach is restricted to situations where Lipschitz continuity of the turning curve $\xi$ is guaranteed for the model at hand, which appears to be a strong restriction on its applicability; this restriction also appears in \cite{andreianov2021}.

\subsection{Abstract framework and general results}

In this paper we propose an existence result elaborated through a fixed-point argument to problem \eqref{eqConsideredModel} under abstract assumptions on $\mathcal{I}$. Roughly speaking, we require that $\mathcal{I}$ maps any admissible solution $\rho$ of the equation \eqref{seqTrueCL} to a Lipschitz continuous turning curve $\xi$. Furthermore, the Lipschitz constant of those turning curves must be uniformly bounded for any $\rho$.

\smallskip
We stress that the Hughes' model with affine cost $c(\rho) = 1 + \alpha \rho$ enters our abstract framework. However, it is not clear whether, for general costs satisfying \eqref{eqCostConditions}, the required Lipschitz bounds hold true. This issue for the original Hughes' model is left for further investigation. Models with more regular dependence of $\xi$ on $\rho$ can be considered as well, including memory and relaxation effects, and for these models the Lipschitz continuity of $\xi$ is justifiable for general costs.

\smallskip
First, let's introduce some notations that will be used throughout the whole paper.

\smallskip\noindent
$\bullet$ We denote $\left\lbrace x < \xi(t) \right\rbrace := \left\lbrace (t,x) \in [0,T]\times\mathbb{R} \textrm{ s.t. } x < \xi(t) \right\rbrace$. Analogously, we use $\left\lbrace x = \xi(t) \right\rbrace$ and $\left\lbrace x > \xi(t) \right\rbrace$.\\[2pt]
 $\bullet$ For any $r>0$, we write
  \[ B_{W^{1,\infty}}(0,r) := \left\lbrace \xi \in W^{1,\infty}((0,T),\mathbb{R}) \textrm{ s.t. } \|\dot{\xi}\|_{\infty} + \|\xi\|_{\infty} \leq r \right\rbrace .\]
$\bullet$ Analogously, we write $B_{L^1}(0,r)$ for the set of $\rho \in L^1((0,T)\times\mathbb{R}, [0,1])$ such that $\|\rho\|_{L^1((0,T)\times\mathbb{R})} \leq r$.

\medskip
In problem \eqref{eqConsideredModel}, $\rho$ is taken as an admissible solution to the discontinuous flux LWR equation \eqref{seqTrueCL}. On the way of proving the existence result, we propose and use a slightly simpler notion of admissible solution for this equation than the notion used in \cite{ElKhatib2013OnEW}, \cite{Amadori2014-uw} and \cite{Amadori2012-na}.
Those notions of solution are equivalent.
\begin{definition}\label{DefSolSansGerme}
Let $\xi \in W^{1,\infty}((0,T))$. Let $\rho_0 \in L^1(\mathbb{R},[0,1])$.
Let $f$  be a concave positive flux such that $f(0) = 0 = f(1)$ and $F(t,x,\rho) :=  \sign(x - \xi(t))f(\rho)$.\\
We say that $\rho \in L^{1}((0,T)\times \mathbb{R},[0,1])$ is an admissible solution to:
\begin{equation}\label{CLeq}
    \left\lbrace \begin{matrix}
    \rho_t + F(t,x,\rho)_x = 0\\
    \rho(t=0,\cdot) = \rho_0(\cdot)
    \end{matrix} \right.
\end{equation}
if
\begin{itemize}
    \item For all $\phi  \in \mathcal{C}^{\infty}_c((0,T)\times \mathbb{R})$,
    \begin{equation}\label{eqSolFaible}
      \iint_{\Omega} \rho \phi_t + F(t,x,\rho) \phi_x \d t \d x = 0
    \end{equation}

    \item For all positive $\phi \in \mathcal{C}^{\infty}_c(\{ x < \xi(t) \}$ (resp. $\phi \in \mathcal{C}^{\infty}_c(\{ x > \xi(t) \})$ ), for all $k \in [0,1]$,
    \begin{equation}\label{eqSolEntropy}
       - \iint_\Omega \left| \rho - k \right|\phi_t + q(\rho,k)\phi_x \d t \d x - \int_{\mathbb{R}} |\rho_0 - k| \phi(0,x) \d x \leq 0
     \end{equation}
    where we set \begin{equation}\label{eqDefDeq}
    q(u,v) := \sign(u-v)\left[ F(t,x,u) - F(t,x,v) \right]
  \end{equation}
\end{itemize}
\end{definition}
Note that the notion of solution makes sense for arbitrary initial datum  $\rho_0 \in L^1(\mathbb{R},[0,1])$ but in order to keep consistency with the standard Hughes' setting, we will restrict our attention to
data $\rho_0$ that fulfill \eqref{eq:HypRho_0}.
\begin{remark}
  Note that in the above definition, no admissibility condition is prescribed at $\{ x = \xi(t) \}$. Only the conservativity (the Rankine-Hugoniot condition following from \eqref{eqSolFaible}) is required at the location of the turning curve.
\end{remark}
\begin{remark}\label{remCancesContinuity}
Definition \ref{DefSolSansGerme} implies that $\rho \in \mathcal{C}^0([0,T], L^1(\mathbb{R}))$. This is proved by an adapted version of the one in \cite{Cances2011-nh}. Such an adapted proof can be found in \cite{Abraham2021}. Remembering this fact makes sense of the notation $\rho(t,\cdot)$ without ambiguity.
\end{remark}
\bigskip
For a given (and fixed) $\xi \in W^{1,\infty}((0,T))$, it is shown this notion of solution gives a well-posed discontinuous flux conservation law in $L^1((0,T) \times \mathbb{R})$ when $\rho_0$ belongs to $L^1(\mathbb{R};[0,1])$. We then define the solver operator:
\begin{equation}\label{eqSolverDef}
  \mathcal{S}_0 : \left\lbrace \begin{matrix}
  W^{1,\infty}((0,T)) \longrightarrow L^1((0,T)\times \mathbb{R})\\
  \xi \mapsto \rho .
\end{matrix} \right.
\end{equation}
This operator $\mathcal{S}_0$ maps $\xi$ a turning curve to $\mathcal{S}_0(\xi) = \rho$ the unique admissible, in the sense of Definition \ref{DefSolSansGerme}, solution to \eqref{seqTrueCL}-\eqref{seqTrueCI} set up in the whole one-dimensional space.

\smallskip

\begin{remark}\label{remFluxesLR}
  The uniqueness of a solution in the sense of Definition \ref{DefSolSansGerme} still holds for
  $$F(t,x,p) :=  \mathbb{1}_{\{x < \xi(t)\}} f_L(p)+\mathbb{1}_{\{x > \xi(t)\}} f_R(p)$$ where
  $f_L$ (resp. $f_R$) is a convex negative (resp. concave positive) flux such that $f_L(0) = f_L(1) = f_R(0) = f_R(1) = 0$. These are the core properties of the fluxes on which rely our proof.
  For instance, modeling a slanted corridor, we can consider $f_{L,R}(\rho) := v_{L,R} \, \rho (1-\rho)$ where $v_L$ and $v_R$ are positive constants accounting for the difference in speed for a pedestrian when moving to the right or the left exit.
\end{remark}

\smallskip

We now present the notion of solution used for the generalized Hughes' model given by system \eqref{eqConsideredModel}. Recalling Remark \ref{remCancesContinuity}, it makes sense for the operator equation \eqref{seqTrueOp} to be verified for all $t \in [0,T]$.
In fact, we will require that $\xi \in W^{1,\infty}((0,T))$ in order to obtain our main result.
We then use the classical embedding result to identify $\xi$ with a unique element of $\mathcal{C}^0([0,T])$.

\begin{definition}\label{defSolToProblem}
  Consider $\mathcal{I} : L^1((0,T) \times \mathbb{R}) \longrightarrow \mathcal{C}^0([0,T])$.
  We say that $(\rho,\xi)$ is a solution to generalized Hughes' model \eqref{eqConsideredModel} if $\rho$ is a solution to \eqref{seqTrueCL}-\eqref{seqTrueCI} in
  the sense of Definition \ref{DefSolSansGerme} and moreover, the equality $\xi = \mathcal{I}(\rho)$ holds in $\mathcal{C}^0([0,T])$.
\end{definition}

Notice that such a solution can be seen as a fixed point of the composed operator $\mathcal{S}_0 \circ \mathcal{I}$.
In order to prove the existence of a solution, we prove a variant of the Schauder's fixed point Theorem (see \cite{Zeidler2012-im}). To be specific, denoting by $\mathcal{I}:\rho \mapsto \xi$ the operator that serves to compute the interface and by $\mathcal{D}: \xi \mapsto \rho$ the one that serves to compute the density, we prove the following statement:

\begin{lemma}\label{thTopoResult}
  Let $(X, \|\cdot\|_{X})$ be a Banach space, $(Y, \|\cdot \|_{Y})$ a metric space and $K$ a compact subset of $Y$.
  Take $\mathcal{D}: (K, \|\cdot\|_Y) \longrightarrow (X, \|\cdot\|_X)$ a continuous operator.
  Assume there exists $B$ a bounded closed convex subset of $X$ such that:
  \begin{subequations}
  \begin{align}
    &\mathcal{I} : (B, \|\cdot\|_X) \longrightarrow (K, \|\cdot\|_Y) \textrm{ is a continuous operator} \label{eqLipStab} \\
    &\mathcal{D} \circ \mathcal{I} (B) \subset B \label{seqWellDefDoI}
  \end{align}
\end{subequations}
  Then $\mathcal{D}\circ\mathcal{I}$ admits a fixed point in $B$.
\end{lemma}

\smallskip
\begin{remark}\label{remWellDefDoI}
  We stress that the assumption \eqref{eqLipStab} implies that, on the subset $B$, $\mathcal{I}$ takes its values in $K$, making $\mathcal{D} \circ \mathcal{I}$ well-defined on $B$.
\end{remark}

The assumptions of Lemma \ref{thTopoResult} permit us to formulate sufficient conditions for the existence of a solution in the sense of Definition \ref{defSolToProblem}.
Specifically, the use of the sets $B_{W^{1,\infty}}(0,r)$ (as $K$) and $\mathcal{C}^0([0,T])$ (as $Y$) is the key to the application of Schauder fixed-point argument to $\mathcal{S}_0 \circ \mathcal{I}$ under reachable assumptions on $\mathcal{I}$ in the Hughes' model framework.

\smallskip
We prove in Section 2 the following proposition saying that $\mathcal{S}_0$ is continuous. This continuity matches with the one required for the operator $\mathcal{D}$ in the above lemma.
\begin{proposition}\label{thContinuityOfS}
  Let $\rho_0$ verify \eqref{eq:HypRho_0}. If $f$ satisfies the non-degeneracy condition:
  \begin{equation}\label{NonDegen}
  \meas\Bigl\{x \in [-\|\rho\|_{\infty};|\rho\|_{\infty}] \textrm{ s.t. }f'(x) = 0 \Bigr\}=0
  \end{equation}
then the solver operator $\mathcal{S}_0 : (W^{1,\infty}((0,T), \|\cdot\|_{\infty}) \longrightarrow (L^1((0,T)\times\mathbb{R}), \|\cdot\|_{L^1((0,T)\times\mathbb{R})})$ is continuous.
\end{proposition}

Combining previous results, we state the main result of this paper:
\begin{theorem}\label{thMainCorollary}
  Let $\rho_0$ verify \eqref{eq:HypRho_0}. Let $B$ a convex closed bounded subset of $L^1((0,T) \times \mathbb{R})$
  and $$\mathcal{I}:  (B,\|\cdot\|_{L^1((0,T)\times \mathbb{R})}) \longrightarrow   (\mathcal{C}^0([0,T], \mathbb{R}), \|\cdot\|_{\infty})$$ be a continuous operator. Assume that $f$ verifies \eqref{NonDegen}.
  If there exists $r>0$ such that:
  \begin{subequations}
  \begin{align}
    &\mathcal{I}(B) \subset B_{W^{1,\infty}}(0,r) \label{seqLipStab} \\
    &\forall \xi \in B_{W^{1,\infty} }(0,r) \textrm{, the unique admissible solution to } \rho_t + \left[ \sign(x-\xi(t))f(\rho)\right]_x = 0 \textrm{ is in } B  \label{seqConditionBpossible}
  \end{align}
\end{subequations}

  then there exists $(\rho,\xi)$ a solution to the problem \eqref{eqConsideredModel} in the sense of Definition \ref{defSolToProblem}.

\end{theorem}

\begin{remark}
One can interpret $B$ as the set where one looks for solutions to \eqref{seqTrueCL}.
\end{remark}

The central point in order to use this theorem is to construct the set B; in below applications, two different choices for $B$ are encountered.

\subsection{Applications}

We search for properties of admissible solution in the sense of Definition \ref{DefSolSansGerme} that are independent of $\xi$. These properties, included in the construction of $B$ must guarantee that $\mathcal{I}(B)$ verifies \eqref{seqLipStab} but also that $B$ is convex, bounded and closed in $L^1((0,T)\times \mathbb{R})$.
In this subsection, we present three applications of Theorem \ref{thMainCorollary}.

\smallskip
First, we consider the operator $\mathcal{I}_0$ associated to the problem \eqref{seqElKhatibEiko} with affine cost function (further detailled in Section 3). Let us exhibit the construction of $B_1$ a set satisfying the conditions \eqref{seqConditionBpossible}-\eqref{seqLipStab} for this choice of $\mathcal{I}$.
Notice that, thanks to the $L^1$-contraction property of the admissible solution $\rho$ that is justified within the uniqueness proof in Section 2, we have:
\begin{align}
  \forall t \in [0,T], \|\rho(t,\cdot) \|_{L^1(\mathbb{R})} \leq \|\rho_0 \|_{L^1(\mathbb{R})} \nonumber \\
  \Rightarrow \|\rho \|_{L^1([0,T] \times \mathbb{R})} \leq T \|\rho_0 \|_{L^1(\mathbb{R})} \label{seqBoundedRho}
\end{align}
Furthermore, we prove that for a certain fixed constant $C > 0$ (which value will be made precise later), for any $\xi \in W^{1,\infty}$, a weak solution to \eqref{seqTrueCL} in the sense \eqref{eqSolFaible} verifies (see Lemma \ref{thControlRhoT} and also \cite{andreianov2021}):
\begin{equation}\label{eqRhoSemiContinuity}
\forall a,b \in \mathbb{R}, \, \forall s,t \in [0,T] \textrm{, }\left| \int_a^b \rho(t,x) - \rho(s,x) \d x \right| \leq C |t-s|.
\end{equation}
Finally, considering an inital datum $0 \leq \rho_0 \leq 1$, we set:
\begin{equation}\label{eqDefBSet}
  B_{1} = \biggl\lbrace
  \rho \in B_{L^1}(0,T \, \|\rho_0\|_{L^1}) \textrm{ s.t. } 0 \leq \rho \leq 1 \textrm{ and } \rho \textrm{ verifies \eqref{eqRhoSemiContinuity}} \biggr\rbrace.
\end{equation}
Applying Theorem \ref{thMainCorollary} with $B_1$ given by \eqref{eqDefBSet} we get:
\begin{proposition}\label{thLinResult}
  Assume that $\mathcal{I}_0:  B_1 \longrightarrow   \mathcal{C}^0([0,T], \mathbb{R})$ is the operator associated with equation \eqref{seqElKhatibEiko} with affine cost $c(\rho) = 1 + \alpha \rho$.
  If $f$ verifies \eqref{NonDegen},
  then there exists $(\rho,\xi)$ a solution to the problem \eqref{eqElKhatibModel} in the sense of Definition \ref{defSolToProblem}.
\end{proposition}

As a second case, we treat ${\mathcal{I}}_\delta$ the operator associated with a modified version of equation \eqref{seqElKhatibEiko} where $\rho$ is replaced by an average density over recent past in equation \eqref{seqElKhatibEiko} (see \eqref{seqModifiedElKhatib}).
This modification is inspired by the use of ``subjective density'' in pedestrian and traffic flows, proposed, e.g., in \cite{Carillo2016} and \cite{Andreianov2014CDA,Donadello2015} (cf. Section~\ref{sec:constrained-exit} where subjective densities are used to model constrained evacuation at exits); this choice introduces inertia effect into agents' perception of the crowd densities. In that setting, we can prove that the image of $\mathcal{I}_{\delta}$ is contained in a bounded subset of $W^{1,\infty}((0,T))$ without requiring the property \eqref{eqRhoSemiContinuity}.
Consequently, we recover the global existence result for any cost $c$ verifying \eqref{eqCostConditions} with the set $B_2$ merely given by:
$$B_2 = \biggl\lbrace \rho \in B_{L^1}(0,T \, \|\rho_0\|_{L^1}) \textrm{ s.t. } 0 \leq \rho \leq 1 \biggr\rbrace.$$

\smallskip
As a third example, we consider $\widetilde{\mathcal{I}_{\epsilon}}$ the operator associated with problem \eqref{seqElKhatibEiko} with a relaxed equilibrium, modeling, in a way different from  $\mathcal{I}_\delta$, inertia effect of the interface dynamics. In this case, the set $B_2$ also
satisfies all the conditions in order to apply Corollary \ref{thMainCorollary}.

\smallskip
Finally, another series of applications (which is an extension of all the previous results to models with different, phenomenologically relevant behavior of agents in exits) is provided in Section~\ref{sec:constrained-exit}.

\subsection{Outline}

In Section 2, we prove the main results of this paper, respectively Theorem \ref{thMainCorollary} and Lemma \ref{thTopoResult}, Proposition \ref{thContinuityOfS}. These proofs hold in an abstract framework where the choice of $\mathcal{I}$ and $B$ are not prescribed.
Then, in Section 3, we detail the construction involving the set $B_1$ satisfying the assumptions of Theorem \ref{thMainCorollary} in the case of $\mathcal{I}_0$ being the operator associated with equation \eqref{seqElKhatibEiko} with affine cost. We also discuss the case of a general cost satisfying \eqref{eqCostConditions} and solve it for the modified operators ${\mathcal{I}}_{\delta}$ and
$\widetilde{\mathcal{I}_{\epsilon}}$ using the set $B_2$.
Eventually, in Section 4, we extend Theorem \ref{thMainCorollary} in a situation with constrained evacuation at exits $x=\pm 1$.

\section{Proof of the main result}

We first deduce Lemma \ref{thTopoResult} from the Schauder fixed-point theorem.

\begin{proof}[Proof of Lemma \ref{thTopoResult}]
We recall that, thanks to condition \eqref{eqLipStab}, $\mathcal{D} \circ \mathcal{I}$ is well defined. What's more, $\mathcal{D}$ and $\mathcal{I}$ are continuous.
So $\mathcal{D} \circ \mathcal{I}$ is continuous from $B$ into itself. Take any subset $A$ of $B$. The set $\mathcal{I}(A) \subset K$ is a relatively compact set in $(Y, \|\cdot\|_Y)$.
Since $\mathcal{D}$ is continuous from $(K, \|\cdot\|_Y)$ into $(X, \|\cdot\|_X)$, $\mathcal{D} \circ \mathcal{I}(A)$ is a relatively compact subset of $X$.
We consequently have $\mathcal{D} \circ \mathcal{I}$ a compact operator from $B$ into itself. Furthermore $B$ is bounded closed convex subset of a Banach space $X$.
We apply Schauder fixed-point theorem (see \cite{Zeidler2012-im}) and conclude to the existence of a fixed point in $B$.
\end{proof}

\smallskip

In order to apply Lemma \ref{thTopoResult} with $\mathcal{D}= \mathcal{S}_0$ the solver associated with the notion of solution of Definition \ref{DefSolSansGerme} ( see \eqref{eqSolverDef} ), we first
need to check that $\mathcal{S}_0$ is well defined from $W^{1,\infty}((0,T))$ into $L^1((0,T)\times \mathbb{R})$ when $\|\rho_0\|_{L^1(\mathbb{R})} < + \infty$.
This is equivalent to well-posedness for the problem \eqref{CLeq}.

\smallskip
We prove below that, thanks to the particular choice of fluxes on each side of the turning curve (emphasized in Remark \ref{remFluxesLR}), Definition \ref{DefSolSansGerme} is restrictive enough to grant uniqueness. This notion of solution is however less restrictive than the one proposed in \cite{ElKhatib2013OnEW,Amadori2012-na}.
It implies that both notions are equivalent, also the existence of such solutions is then directly inherited from the proof found in \cite{Amadori2012-na}.
Note that one can prove the existence result for our notion of solution through the convergence of a finite volume scheme (we do so in Section 4, in the context of flux-limited exit behavior at the exits $x= \pm 1$).

\begin{theorem}\label{thPreuveUnicite}
Let $ \rho $,$\hat{\rho}$ be two entropy solutions in the sense of Definition \ref{DefSolSansGerme} with initial datum $\rho_0$ (resp. $\hat{\rho}_0$).
Let $L_{f}$ be the lipschitz constant of $f$.
If $\xi \in W^{1,\infty}((0,T))$, we have:
\[ \textrm{for a.e. } t \in [0,T], \forall a,b \in \mathbb{R} \textrm{, } \int_a^b |\rho(t,x) - \hat{\rho}(t,x)| dx \leq \int_{a-L_f t}^{b+L_f t} |\rho^0(x) - \hat{\rho}^0(x)| dx. \]
In particular, there exists at most one entropy solution associated to a given initial datum $\rho_0$.
\end{theorem}

In order to prove this Theorem, we introduce notation for the right and left strong traces of $\rho$ along a Lipschitz curve $\xi$. Let $\xi \in W^{1,\infty}((0,T), \mathbb{R})$. Then, $\gamma_{L} \rho (t) \in L^{\infty}((0,T))$ (resp. $\gamma_{R} \rho (t)$ ) is such that,
for any $\phi \in \mathcal{C}^0([0,1])$,
$$ \ess \lim_{\epsilon \rightarrow 0^+} \frac{1}{\epsilon} \int_0^T \int_{\xi(t) - \epsilon}^{\xi(t)} \left| \phi(\rho(t,x)) - \phi(\gamma_L \rho(t)) \right| \d x \d t = 0 $$
$$\left( \textrm{ respectively, } \ess \lim_{\epsilon \rightarrow 0^+} \frac{1}{\epsilon} \int_0^T \int_{\xi(t)}^{\xi(t) + \epsilon} \left| \phi(\rho(t,x)) - \phi(\gamma_R \rho (t)) \right| \d x \d t = 0 \right) $$
The existence of those traces is proven in \cite{Vasseur2001-ig}.
\begin{remark}
Generalization of the approach of the present paper to general cost function $c$, for the original Hughes' model, may require going below the Lipschitz regularity of $\xi$. In this respect, let us point out that extension of the above uniqueness claim to $W^{1,1}$ regularity of $\xi$ is feasible, while weakening the regularity of $\xi$ even more presents a serious difficulty for the theory of discontinuous-flux conservation laws \cite{AndrKarlRis-2011}.
\end{remark}

\begin{proof}[Proof of Theorem \ref{thPreuveUnicite}]
Remembering Remark \ref{remFluxesLR} and for a more comprehensive presentation of the proof, we denote $f_R = f$ and $f_L = -f$.\\
To main idea of the proof consists of using Kruzkhov's doubling variable technique (see \cite{Evans1998-au}) on each side of the curve $\left\lbrace x = \xi(t)
\right\rbrace$. Since $\xi$ is Lipschitz continuous we can join both pieces getting left and right traces along this turning curve, following the general approach as in \cite{AndrKarlRis-2011,Andreianov2014CDA}. We get, for any $\phi \in \mathcal{D}^+$,
\begin{equation}\label{eqCenterTerme}\tag{\textasteriskcentered}
   - \iint_{\Omega} |\rho-\hat{\rho}| \phi_t + q(\rho,\hat{\rho}) \phi_x
\leq \int_0^T \phi(t,\xi(t)) \left[ q_R(\gamma_R \rho, \gamma_R \hat{\rho} )
- q_L(\gamma_L \rho, \gamma_L \hat{\rho}) \right]
\end{equation}
where \( q_{L,R}(\rho,\hat{\rho}) := \sign(\rho-\hat{\rho})\left[ f_{L,R}(\rho) - f_{L,R}(\hat{\rho}) - \dot{\xi}(t)(\rho - \hat{\rho}) \right] \).

\smallskip
On another side, using traces' existence, we also recover from \eqref{eqSolFaible} the Rankine-Hugoniot condition:
\begin{empheq}{align}\label{eqRH} \tag{$**_{\textstyle \rho}$}
\textrm{for a.e. } t \in (0,T) \textrm{, }
f_R(\gamma_R \rho (t)) - \dot{\xi}(t) \gamma_R \rho(t) = f_L(\gamma_L \rho (t) ) - \dot{\xi}(t) \gamma_L \rho(t)
\end{empheq}
We also have the analogous relation for $\hat{\rho}$ that we denote ($**_{\textstyle \hat\rho}$).

\smallskip
Fix $t \in (0,T)$ such that \eqref{eqRH} and ($**_{\textstyle \hat\rho}$) are true.
We denote the set of values for $\gamma_L \rho$ (resp. $\gamma_R \rho$) that verify \eqref{eqRH}:
\[ \Gamma^{L,R} := \left\lbrace a \in \mathbb{R} \textrm{ s.t. } \exists b \in \mathbb{R}, f_{L,R}(a) - \dot{\xi}(t)a = f_{L,R}(b) - \dot{\xi}(t)b \right\rbrace. \]
Due to the particular choice of the pair of fluxes $(f_L,f_R)$, those sets are non-empty. Its geometries are pictured below.

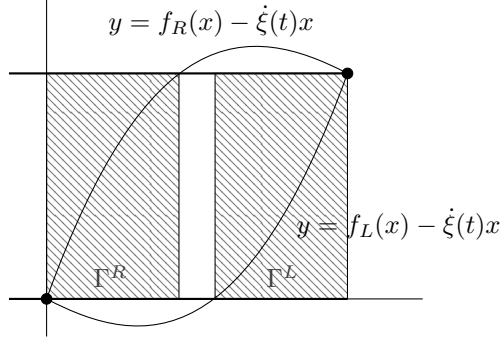
\setlength{\unitlength}{1cm}

\begin{tikzpicture}
  \draw[pattern=north west lines, pattern color=gray] (0,0) rectangle (1.76,3);
  \draw[pattern=north west lines, pattern color=gray] (2.24,0) rectangle (4,3);

  \draw (-0.5,0) -> (5,0);
  \draw (0,-0.5) -> (0,4);
  \filldraw [black] (0,0) circle (2pt);
  \filldraw [black] (4,3) circle (2pt);
  \filldraw [black, thick] (-0.5,0) -> (4,0);
  \filldraw [black,thick] (-0.5,3) -> (4,3);
  \filldraw[darkgray] (0.5,0.3) node[anchor=west]{$\Gamma^R$};
  \filldraw[darkgray] (2.8,0.3) node[anchor=west]{$\Gamma^L$};

  \draw (0,0) .. controls (1.33,3.7) and (2.66,3.7) .. (4,3);
  \draw (0,0) .. controls (1.33,-0.7) and (2.66,-0.7)  .. (4,3);
  \filldraw[black] (3.2,1) node[anchor=west]{$y = f_L(x) - \dot{\xi}(t)x$};
  \filldraw[black] (0.7,3.7) node[anchor=west]{$y = f_R(x) - \dot{\xi}(t)x$};
\end{tikzpicture}

Recalling the properties of $f_L$ and $f_R$ emphasized in Remark \ref{remFluxesLR} and using the signs of $f_L'$ and $f_R'$, we let the reader verify that, for any $\dot{\xi}(t)$, $x \mapsto f_R(x) - \dot{\xi}(t)x$ has the same monotonicity on $\Gamma_R$ as $x \mapsto f_L(x) - \dot{\xi}(t)x$ on $\Gamma_L$.

\smallskip
Consequently, if $(\gamma_L \rho, \gamma_R \rho)$ verifies \eqref{eqRH} and $(\gamma_L \hat{\rho}, \gamma_R \hat{\rho})$ verifies ($**_{\textstyle \hat\rho}$),

\smallskip
\begin{itemize}
\item $\sign(\gamma_R \rho- \gamma_R \hat{\rho})\sign\left(f_R(\gamma_R \rho) - f_R(\gamma_R \hat{\rho}) - \dot{\xi}(t)(\gamma_R \rho  -\gamma_R \hat{\rho}) \right)$\\
 \hspace*{10pt}$= \sign(\gamma_L \rho- \gamma_L \hat{\rho})\sign\left(f_L(\gamma_L \rho) - f_L(\gamma_L \hat{\rho}) - \dot{\xi}(t)(\gamma_L \rho - \gamma_L \hat{\rho}) \right)$

 \smallskip
\item \eqref{eqRH}-($**_{\textstyle \hat\rho}$) implies that \\
$f_R(\gamma_R \rho) - f_R(\gamma_R \hat{\rho}) - \dot{\xi}(t)(\gamma_R \rho  - \gamma_R v)  = f_L(\gamma_L \rho) - f_L(\gamma_L \hat{\rho}) - \dot{\xi}(t)(\gamma_L \rho - \gamma_L \hat{\rho}). $

\end{itemize}

\medskip
Therefore we have:
\[ \textrm{for a.e. } t \in (0,T) \textrm{,  } \, \, q_R(\gamma_R \rho, \gamma_R \hat{\rho}) - q_L(\gamma_L \rho, \gamma_L \hat{\rho}) = 0. \]

Consequently, from \eqref{eqCenterTerme}, we recover the global Kato's inequality: for any $\phi \in \mathcal{D}^+(\Omega)$,
\[ - \iint |\rho-\hat{\rho}| \phi_t + q(\rho,\hat{\rho}) \phi_x
\leq 0. \]

The remaining arguments are identical to the classical framework of Kruzkhov. Integrating on the trapezoid \( \mathbb{1}_{[0,t]}(s)\mathbb{1}_{[a-L_f(t-s),b+L_f(t-s)]}(x) \), $L_f$ being the Lipschitz constant of $f$, we get the localized $L^1$ contraction property:
\begin{equation}\label{eqL1Contract}
   \int^b_a |\rho(t,x) -\hat{\rho}(t,x)| dx \leq \int^{b +L_f t}_{a - L_f t} |\rho(0,x) - \hat{\rho}(0,x)| dx.
 \end{equation}
\end{proof}

Consequently, the solver operator $\mathcal{S}_0$ is well defined from $W^{1,\infty}((0,T))$ into $L^1((0,T)\times \mathbb{R})$.
In order to apply Lemma \ref{thTopoResult} with $\mathcal{D} = \mathcal{S}_0 : \left(W^{1,\infty}((0,T)), \|\cdot \|_{\infty} \right) \longrightarrow \left( L^1((0,T)\times \mathbb{R}), \|\cdot \|_{L^1((0,T)\times \mathbb{R})} \right) $, we also show the continuity of this operator.
Let's denote for any $a<b\in \mathbb{R}, s < t \in [0,T]$, the trapezoid:
\begin{equation}\label{eqDefTrapez}
  \mathcal{T}_{a,b}^{s,t} := \bigl\{ (\tau,x) \in (0,T)\times \mathbb{R} \textrm{ s.t. }
  \tau \in [s,t], \, x \in ( a + (\tau-s) L_f \, , \, b - (\tau-s) L_f ) \bigr\},
  \end{equation}
  where $L_f$ is the Lipschitz constant of $f$.
We isolate the following useful lemma that comes from \eqref{eqL1Contract}.
\begin{lemma}\label{thEqualityTrapezoid}
  Let $\rho_0$ satisfy \eqref{eq:HypRho_0}, $\xi \in W^{1,\infty}((0,T))$ and $\rho$ be the entropy solution in the sense of Definition \ref{DefSolSansGerme} to \eqref{CLeq} on $(0,T) \times \mathbb{R}$.
  Denote $\hat{\rho}$ the Kruzhkov entropy solution on $(s,t)\times \mathbb{R}$ to
  \footnote{Here $\rho(s,\cdot)$ is understood in view of $s$ being a Lebesgue's point of $\rho \in L^{\infty}((0,T), L^1(\mathbb{R}))$. Recalling Remark \ref{remCancesContinuity}, this is in fact true for any $s \in [0,T]$.}
  $$\left\lbrace \begin{matrix}
  \hat{\rho}_t + f(\hat{\rho})_x = 0 \\
  \hat{\rho}(s, \cdot) = \rho(s, \cdot) \mathbb{1}_{(a,b)}(\cdot).
  \end{matrix} \right.$$

  Then, for any $a<b\in \mathbb{R}, s < t \in [0,T]$, there holds
  \begin{equation}\label{eqEqualityInL}
    \mathcal{T}_{a,b}^{s,t} \subset \{ x > \xi(t) \} \;\Longrightarrow\; \rho = \hat{\rho} \textrm{ a.e.} \text{ on $\mathcal{T}_{a,b}^{s,t}$}.
    \end{equation}
\end{lemma}

\begin{proof}
  This lemma immediatly follows from \eqref{eqL1Contract}.
\end{proof}

We now prove Proposition \ref{thContinuityOfS} using this lemma.

\begin{proof}[Proof of Proposition \ref{thContinuityOfS}]
Consider $(\xi_n)_{n \in \mathbb{N}}$ and $ \xi \in W^{1,\infty}((0,T))$ such that $\|\xi_n - \xi\|_{\infty} \longrightarrow 0$. We denote $\rho_n := \mathcal{S}_0(\xi_n)$.
Let $K$ a compact subset of $\left\{x > \xi(t)\right\}$. Let $\epsilon > 0$ such that $K \subset \left\{ x > \xi(t) + \epsilon \right\}$.

\smallskip
We cover $K$ by a finite number of trapezoids of the form \eqref{eqDefTrapez}. Without loss of generality we can suppose that each trapezoid is contained in $\left\{ x > \xi(t) + \epsilon \right\}$:
$$ K \subset \bigcup_{i \in I} \mathcal{T}_{a_i,b_i}^{s_i,t_i} \subset \left\{ x > \xi(t) + \epsilon \right\}, \, \card(I) < + \infty.$$

Since $\|\xi_n - \xi\|_{\infty} \longrightarrow 0$, for any
$\epsilon > 0$, there exists $n_0 \in \mathbb{N}$ such that $\forall t \in [0,T], \, n \geq n_0
\Rightarrow |\xi_n(t) - \xi(t)| \leq \epsilon $.\\
This implies $\xi_n(t) \in [ \xi(t) - \epsilon ; \xi(t) + \epsilon]$. Then,
\begin{align}
  \forall x \in \mathbb{R} \backslash [ \xi(t) - \epsilon ; \xi(t) + \epsilon] \textrm{, } \sign(x - \xi_n(t)) = \sign(x-\xi(t)). \label{interesEqSign}
\end{align}
Then, for such a $n_0$, for any $n \geq n_0$, each trapezoid $\mathcal{T}_{a_i,b_i}^{s_i,t_i} \subset \{ x > \xi_n(t) \}$.
Using Lemma \ref{thEqualityTrapezoid}, for any $n \geq n_0$, $\rho_n$ is equal almost everywhere in $\mathcal{T}_{a_i,b_i}^{s_i,t_i}$ to the Kruzhkov entropy solution of:
$$\left\lbrace \begin{matrix}
\rho_t + f(\rho)_x = 0 \\
\rho(s_i, \cdot) = \rho_n(s_i, \cdot) \mathbb{1}_{(a_i,b_i)}(\cdot).
\end{matrix} \right.$$
We are now in a position to apply the averaging compactness lemma (see Theorem 5.4.1 in \cite{Perthame2003-au}) on the trapezoid $\mathcal{T}_{a_0,b_0}^{s_0,t_0}$. We get a subsequence $(\rho_{n_k})_{k \in \mathbb{N}}$ that converges in $L^1(\mathcal{T}_{a_0,b_0}^{s_0,t_0})$.
We then apply the averaging compactness lemma with $(\rho_{n_k})_k$ on $\mathcal{T}_{a_1,b_1}^{s_1,t_1}$. Repeating this process for each $i \in I$, we recover a subsequence $(\rho_{n_j})_j$ that converges in $L^1(\bigcup_{i \in I} \mathcal{T}_{a_i,b_i}^{s_i,t_i})$. Then $(\rho_{n_j})_j$ converges in $L^1(K)$.\\
 To conclude, we point out that this reasoning holds for any $K \subset \left\{x > \xi(t)\right\}$. This is also true for compact subsets of $\left\{x < \xi(t)\right\}$. Since $\xi$ is Lipschitz, $\meas(\{x =\xi(t)\})=0$. Consequently there exists a subsequence $(\rho_{n_k})$ that converges almost everywhere on $(0,T)\times\mathbb{R}$ and in $L^1_{loc}((0,T)\times\mathbb{R})$.
 Moreover, we have $\rho_{n_k} \longrightarrow \rho$ in $L^1((0,T)\times\mathbb{R})$ because for $[a,b]\cap [-1,1]=\emptyset$, $\rho_n=0$ on $\mathcal{T}_{a,b}^{0,T}$ , due to the choice of $\rho_0$ verifying \eqref{eq:HypRho_0}.

 \smallskip
 	 Now, $\rho$ is actually $\mathcal{S}_0(\xi)$. Indeed,
 recall that $\rho$ has no admissibility condition to satisfy on $\{x=\xi(t)\}$ beyond the Rankine-Hugoniot relation. Then, we can pass to the limit in the entropy inequalities \eqref{eqSolEntropy} (where, for $n$ large enough, the support of the test function does not intersect the curve $\{x=\xi_n(t)\}$ for $t\in [0,T]$) and pass to the limit in \eqref{eqSolFaible} by dominated convergence.

 \smallskip
 This reasoning can be reproduced for any subsequence of $(\rho_n)_n$. Thanks to a classical argument of compacity, if any converging subsequence $(\mathcal{S}_0(\xi_{n_k}))_{k \in \mathbb{N}}$ converges to $\mathcal{S}_0(\xi)$, the whole sequence $(\mathcal{S}_0(\xi_n))_n$ converges in $L^1$ to $\mathcal{S}_0(\xi)$.
 So $\mathcal{S}_0: (W^{1,\infty}((0,T)),\|\cdot\|_{\infty}) \longrightarrow (L^1((0,T)\times\mathbb{R}),\|\cdot\|_{L^1((0,T)\times\mathbb{R})})$ is continuous.
\end{proof}

\smallskip

We now combine all the previous results to get existence of a solution in the sense of Definition \ref{defSolToProblem}.
\begin{proof}[Proof of Theorem \ref{thMainCorollary}]
Suppose there exists $r>0$ such that \eqref{seqLipStab}-\eqref{seqConditionBpossible} are verified.\\
Using the notations of Theorem \ref{thTopoResult} we take:

\smallskip
 \hspace*{10pt}$\bullet$\; $Y = (\mathcal{C}^0([0,T]), \|\cdot\|_{\infty})$\\
 \hspace*{10pt}$\bullet$\; $X = (L^1((0,T)\times \mathbb{R}), \|\cdot\|_{L^1((0,T)\times \mathbb{R})})$\\
  \hspace*{10pt}$\bullet$\; $K$ as the compact set of $\mathcal{C}^0([0,T])$ obtained as the image of $B_{W^{1,\infty}}(0,r)$ under the standard embedding.

\smallskip
Using Proposition \ref{thContinuityOfS} and Theorem \ref{thPreuveUnicite}, we know that $\mathcal{S}_0 : (K, \|\cdot\|_{Y}) \longrightarrow (X, \|\cdot\|_{X})$ is well defined and continuous.
Further, notice that condition \eqref{seqLipStab} is equivalent to \eqref{eqLipStab} and that condition \eqref{seqConditionBpossible} implies \eqref{seqWellDefDoI}.
We are now in a position to use Lemma \ref{thTopoResult}. We conclude to the existence of a solution to \eqref{eqConsideredModel} in the sense of Definition \ref{defSolToProblem}.
\end{proof}

\section{Lipschitz continuity of the turning curve: examples}\label{sec:applications}

In this section, we will enumerate examples of the abstract problem \eqref{eqConsideredModel}
\begin{subequations}
\begin{empheq}[left = \empheqlbrace]{align*}
\rho_t + \left[\sign(x-\xi(t))f(\rho)\right]_x &= 0 \\
\rho(0,x) &= \rho_0(x)  \\
\xi &=\mathcal{I}(\rho),
\end{empheq}
\end{subequations}
where we can construct a set $B$ such that the prescribed operator $\mathcal{I}$ satisfies the required properties in order to apply Theorem \ref{thMainCorollary}; this includes the original Hughes' model \eqref{eqElKhatibModel}  with affine costs and its modifications, taking into account time-inertia effects and allowing for general costs. Note that further examples, with modified exit conditions, are considered in Section~\ref{sec:constrained-exit}. For such examples, we exhibit the construction of this set. Consequently, we get existence of a solution in the sense of Definition \ref{defSolToProblem} in those situations.

\subsection{Hughes's model with affine cost}
We first consider the model \eqref{eqElKhatibModel}:
\begin{subequations}
\begin{empheq}[left = \empheqlbrace ]{align*}
\rho_t + \left[ \sign(x-\xi(t))\rho v(\rho)\right]_x  =  0 \\
\int_{-1}^{\xi(t)} c(\rho(t,x))dx  = \int_{\xi(t)}^1 c(\rho(t,x))dx,
\end{empheq}
\end{subequations}
with initial datum satisfying \eqref{eq:HypRho_0} where we choose, for some $\alpha>0$,
\begin{equation}\label{eq:affinecost}
	c(p) = 1 + \alpha p.
\end{equation}

First, let us recall the definition of the set $B_1$ constructed in the introduction:
\begin{equation}
  B_{1} = \biggl\lbrace
  \rho \in B_{L^1}(0,T \, \|\rho_0\|_{L^1}) \textrm{ s.t. } 0 \leq \rho \leq 1 \textrm{ and } \rho \textrm{ verifies \eqref{eqRhoSemiContinuity}} \biggr\rbrace. \tag{\ref{eqDefBSet}}
\end{equation}
In this setup, we have the following proposition:
\begin{proposition}\label{thCoutLineaireLS}
Assume the cost is given by \eqref{eq:affinecost}. Then the following properties hold:
  \begin{enumerate}
\item For any $\xi \in W^{1,\infty}((0,T))$, $\mathcal{S}_0 (\xi) \in B_1$.
 \item There exists $r >0$ such that, for any $\rho \in B_1$,
there exists a unique solution $\xi \in B_{W^{1,\infty}}(0,r)$ to \eqref{seqElKhatibEiko}.
We denote $\mathcal{I}_0$ the operator that maps $\rho \in B_{1}$ to $\xi$ the unique solution to \eqref{seqElKhatibEiko}. Consequently, this operator is well defined and monovaluated.
\item $\mathcal{I}_0 :  (B_1 , \|\cdot\|_{L^1((0,T) \times \mathbb{R})}) \longrightarrow
(W^{1,\infty}([0,T]), \|\cdot\|_{\infty})$ is continuous.
\item $B_1$ is closed convex and bounded in $L^1((0,T)\times \mathbb{R})$.
\end{enumerate}
\end{proposition}

Consequently, $\mathcal{I}_0$
verifies \eqref{seqLipStab}-\eqref{seqConditionBpossible} for the set $B_1$.
We apply Theorem \ref{thMainCorollary} and get the desired existence of a solution for the problem \eqref{eqElKhatibModel} with affine cost \eqref{eq:affinecost}. That proves Proposition \ref{thLinResult}.

\smallskip
In order to prove of Proposition \ref{thCoutLineaireLS}, we rely on two lemmas that we chose to isolate in order to use them in the other examples.
\begin{lemma}\label{thControlRhoT}
  Let $a, b \in \mathbb{R}$, $a<b$. Let $s,t \in [0,T]$, $s<t$. Fix $\xi \in W^{1,\infty}((0,T))$.
  We denote $\rho$ a solution in the sense of Definition \ref{DefSolSansGerme}.
  Then, there exists $C > 0$, independent of $a,b,s,t, \xi$ and $\rho$, such that:
  \begin{equation}\label{eqControlRhoT}
    \left| \int_a^b \rho(t,x) - \rho(s,x) \d x \right| \leq C |t-s|.
  \end{equation}
  We recall that there's no ambiguity in considering $\rho(t,.)$ since $\rho \in \mathcal{C}^0([0,T],L^1(\mathbb{R}))$ (see Remark \ref{remCancesContinuity}).
\end{lemma}
\begin{proof}[Proof of Lemma \ref{thControlRhoT}]
Let $(\kappa_n)_{n \in \mathbb{N}}$ be a mollifier. We set
\[ \Psi(\tau,x) := \mathbb{1}_{[a,b]}(x) \mathbb{1}_{[s,t]}(\tau) \, \textrm{ and }\, \phi(\tau,x) := \Psi \ast \kappa_n (\tau,x). \]
Using $\phi$ as test function in \eqref{eqSolFaible}, making $n \longrightarrow + \infty$ we get:
\[ \int_a^b \rho(s,x) - \rho(t,x) \d x + \int_s^t F(\tau,a,\rho(\tau,a)) - F(\tau,b,\rho(\tau,b)) \d \tau = 0 \]
Consequently,
\begin{align*}
  \left| \int_a^b \rho(t,x) - \rho(s,x) \d x \right|  \leq \left| \int_s^t F(\tau,a,\rho(\tau,a)) - F(\tau,b,\rho(\tau,b)) \d \tau \right|
   \leq \left( 2 \sup_{p \in [0,1]}|f(p)|\right) \left| t- s \right|
\end{align*}
\end{proof}

\begin{lemma}\label{thLipEstimate}
  Let $s < t \in [0,T]$. Let $\xi$ be a solution to \eqref{seqElKhatibEiko}.
  We denote $\ubar{\xi} := \min(\xi(t),\xi(s))$ and $\bar{\xi} := \max(\xi(t),\xi(s))$. Then
  \begin{equation}\label{eqLipEstimate}
2 \left| \xi(t) - \xi(s) \right| \leq \left| \int_{-1}^{\ubar{\xi}} c(\rho(t,x)) - c(\rho(s,x)) \d x - \int_{\bar{\xi}}^{1} c(\rho(t,x)) - c(\rho(s,x)) \d x \right|
  \end{equation}
\end{lemma}

\begin{proof}[Proof of Lemma \ref{thLipEstimate}]
We first treat the case $\xi(s) \leq \xi(t)$.\\
We have:
\begin{align*}
\int_{-1}^{\xi(s)} c(\rho(s,x)) \d x = \int_{\xi(s)}^{\xi(t)} c(\rho(s,x)) \d x + \int_{\xi(t)}^1 c(\rho(s,x)) \d x \\
\int_{-1}^{\xi(s)} c(\rho(t,x)) \d x = - \int_{\xi(s)}^{\xi(t)} c(\rho(t,x)) \d x + \int_{\xi(t)}^1 c(\rho(t,x)) \d x
\end{align*}
If we substract both equalities,
\[ \int_{\xi(s)}^{\xi(t)} c(\rho(s,x)) + c(\rho(t,x)) \d x =
\int_{-1}^{\xi(s)} c(\rho(s,x)) - c(\rho(t,x)) \d x - \int_{\xi(t)}^{1} c(\rho(s,x)) - c(\rho(t,x)) \d x \]
On the contrary, if $\xi(s) \geq \xi(t)$, with an analogous argument we get:
\[ \int_{\xi(t)}^{\xi(s)} c(\rho(s,x)) + c(\rho(t,x)) \d x =
\int_{-1}^{\xi(t)} c(\rho(t,x)) - c(\rho(s,x)) \d x - \int_{\xi(s)}^{1} c(\rho(t,x)) - c(\rho(s,x)) \d x \]
Using the fact that $c \geq 1$ we get:
\[ 2|\xi(t)-\xi(s)| = 2 (\bar{\xi} - \ubar{\xi}) \]
\[\leq
\int_{\ubar{\xi}}^{\bar{\xi}} c(\rho(s,x)) + c(\rho(t,x)) \d x \leq \left|
\int_{-1}^{\ubar{\xi}} c(\rho(s,x)) - c(\rho(t,x)) \d x - \int_{\bar{\xi}}^{1} c(\rho(s,x)) - c(\rho(t,x)) \d x \right| \]
\end{proof}

We are now ready to prove Proposition \ref{thCoutLineaireLS}.
 \begin{proof}[Proof of Proposition \ref{thCoutLineaireLS}]
First, consider $\rho_0$ satisfying \eqref{eq:HypRho_0}. Using $\hat{\rho} = 0$ in \eqref{eqL1Contract},
we prove that for all $t$ in $ [0,T]$, $\|\rho(t,\cdot) \|_{L^1(\mathbb{R})} \leq \|\rho_0\|_{L^1(\mathbb{R})}$. This readily yields:
\begin{align}
  \|\rho\|_{L^1([0,T] \times \mathbb{R})} \leq T \|\rho_0 \|_{L^1(\mathbb{R})} \tag{\ref{seqBoundedRho}}.
\end{align}
Combining this result with Lemma \ref{thControlRhoT}, we prove the first assertion of Proposition \ref{thCoutLineaireLS}.

\smallskip
Second, fix $\rho \in B_1$. We prove existence and uniqueness of $\xi \in L^{\infty}([0,T])$ satisfying \eqref{seqElKhatibEiko} for any $t\in [0,T]$.\\
Let $t \in [0,T]$, we set:
\begin{align*}
  \Psi^+(a) := \int_{-1}^a c(\rho(t,x)) \d x, \,\,
  \Psi^-(a) := \int_{a}^1 c(\rho(t,x)) \d x.
\end{align*}
One can notice that, because $c > 0$, $\Psi^+$ is a continuous strictly increasing function, while $\Psi^-$ is continuous and strictly decreasing on $[-1,1]$. Therefore, $a \mapsto \Psi^+(a) - \Psi^-(a)$ is continuous, strictly increasing, negative at $a=-1$ and positive at $a=1$. Consequently, there exists only one $\tilde{a} \in (-1,1)$ such that $\Psi^+(\tilde{a} ) = \Psi^-(\tilde{a} )$.
This can be done for any $t \in [0,T]$. Consequently, we get existence and unicity of $\xi \in L^{\infty}$.

\smallskip
We now prove that $\xi \in W^{1,\infty}([0,T])$. Using Lemma \ref{thLipEstimate} we get:
\begin{align*}
  2 \left| \xi(t) - \xi(s) \right| &\leq \left| \int_{-1}^{\ubar{\xi}} c(\rho(t,x)) - c(\rho(s,x)) \d x - \int_{\bar{\xi}}^{1} c(\rho(t,x)) - c(\rho(s,x)) \d x \right|\\
  &\leq \alpha \left| \int_{-1}^{\ubar{\xi}} \rho(t,x) - \rho(s,x) \d x \right| + \alpha \left| \int_{\bar{\xi}}^{1} \rho(t,x) - \rho(s,x) \d x \right|
\end{align*}
And using Lemma \ref{thControlRhoT}, with the choice 
\eqref{eq:affinecost} of the cost, we get:
\[ 2 \left| \xi(t) - \xi(s) \right|
\leq 2 \alpha C \left| t-s \right| \]
We conclude that taking $r = \alpha C$, one guarantees that $\xi$ is always in $B_{W^{1,\infty}}(0,r)$.

\smallskip
 We now prove the continuity of the operator $\mathcal{I}_0$.
Let's consider $\rho$, $\rho_n \in B_1$. Then, for a given $t \in [0,T]$, using \eqref{seqElKhatibEiko} for both $\xi := \mathcal{I}_0(\rho)$ and $\xi_n := \mathcal{I}_0(\rho_n)$, we recover:
$$\int_{\xi_n(t)}^{\xi(t)} c(\rho) + \int_{-1}^{\xi_n(t)} c(\rho) - \int_{-1}^{\xi_n(t)} c(\rho_n)
= \int_{\xi(t)}^{\xi_n(t)} c(\rho) + \int_{\xi_n(t)}^1 c(\rho) - \int_{\xi_n(t)}^{1}c(\rho_n)$$
And rearranging the integrals, we get:
$$ 2 \int_{\xi_n(t)}^{\xi(t)}c(\rho) = \int_{-1}^{1}\left[ c(\rho)-c(\rho_n) \right] \sign(x - \xi_n(t)).
$$ \label{eqCalculusOfPsi}
Notice that
\begin{align*}
  \int_0^T |\xi - \xi_n|
  \begin{multlined}[t]
   \leq \int_0^T \left| \int_{\xi(t)}^{\xi_n(t)} c(\rho) \right|
    \leq \frac{1}{2} \int_0^T \left| \int_{-1}^{1} \sign(x - \xi_n(t))  \left[c(\rho)-c(\rho_n)\right] \right| \\
    \leq \frac{1}{2} \int_0^T \int_{-1}^{1}  \left|c(\rho)-c(\rho_n) \right|
    \leq \frac{\alpha}{2} \int_0^T \int_{-1}^{1}  \left|\rho - \rho_n \right|.
  \end{multlined}
\end{align*}
Consequently, if $\| \rho - \rho_n \|_{L^1((0,T)\times \mathbb{R})} \longrightarrow 0$,
$$ \|\xi - \xi_n\|_{L^1((0,T))} \longrightarrow 0.$$
We recall, that $\xi, \xi_n \in \mathcal{I}_0(B_1)$ are $r$-Lipschitz.
On any open subset of $[0,T]$ there exists a point $t$ where the continuous function $\xi(\cdot) - \xi_n(\cdot)$ is less or egal to its $L^1$-average. Using the fact that $[0,T]$ can be covered by a finite $\epsilon$-network and that the derivative of $\xi(\cdot) - \xi_n(\cdot)$ is bounded on this network,
we recover that $\|\xi-\xi_n\|_{\infty} \longrightarrow 0$ when $\| \rho - \rho_n \|_{L^1((0,T)\times \mathbb{R})} \longrightarrow 0$.
This proves the third point of Proposition \ref{thCoutLineaireLS}.

\smallskip
Eventually, let $\rho_1, \rho_2 \in B_1$, $\lambda \in [0,1]$; it is readily checked that $\lambda \rho_1 + (1-\lambda)\rho_2$ still satisfies \eqref{eqControlRhoT}. Then $B_1$ is convex.
It is also readily checked that we can pass to the $L^1((0,T)\times \mathbb{R})$ limit in \eqref{eqControlRhoT}, proving that $B_1$ is closed.
By construction $B_1$ is bounded. That ends the proof of Proposition \ref{thCoutLineaireLS}.
 \end{proof}

\subsection{The general cost case evaluated for a subjective density}

In the same setup \eqref{eqElKhatibModel}, let's further prospect the situation for a cost function $c$ verifying \eqref{eqCostConditions}.
Most of the items of Proposition \ref{thCoutLineaireLS} hold with the set $B_1$. The first point is independent of the nature of $c$. The third point proof still holds with general cost if the second point holds.
Proof of existence and unicity of $\xi \in L^{\infty}((0,T))$ is still valid. In fact, the main issue lies in proving that $\xi$ is Lipschitz for any $\rho$ in a given set $B$. \\
In order to explore this issue, let's start from Lemma \ref{thLipEstimate} estimate \eqref{eqLipEstimate}:
$$2 \left| \xi(t) - \xi(s) \right| \leq \left| \int_{-1}^{\ubar{\xi}} c(\rho(t,x)) - c(\rho(s,x)) \d x - \int_{\bar{\xi}}^{1} c(\rho(t,x)) - c(\rho(s,x)) \d x \right|$$
Recall that $c$ satisfies \eqref{eqCostConditions}.
We set $\bar{\alpha} := \essup_{u \in [0,1]} c'(u)$,
$\ubar{\alpha} := \essinf_{u \in [0,1]} c'(u) > 0$.
Using the negative and positive parts of $(\rho(t,\cdot)-\rho(s,\cdot))$, rearranging the terms we get the following estimate:
\begin{align}
  2 \left| \xi(t) - \xi(s) \right| \leq &\left( \frac{\bar{\alpha}+\ubar{\alpha}}{2} \right) \left| \int_{-1}^{\ubar{\xi}} \rho(t,x) - \rho(s,x) \d x
  - \int_{\bar{\xi}}^{1} \rho(t,x) - \rho(s,x) \d x \right|\notag \\
  &+ \left( \frac{\bar{\alpha}-\ubar{\alpha}}{2} \right) \int_{-1}^{1} \left| \rho(t,x) - \rho(s,x) \right| \d x =: I_1 + I_2 \label{eqL1ContinuityCalculation}
\end{align}
The first term $I_1$ of the right member is controlled by the estimate of Lemma \ref{thControlRhoT}. The issue lies in controlling the second term $I_2$.
This suggests that, in order to prove that $\xi \in W^{1,\infty}((0,T))$ we need an estimate of the modulus of continuity of $\rho$ as an element of $\mathcal{C}^0([0,T], L^1(\mathbb{R}))$. While the standard Oleinik regularizing effect can be used locally away from the turning curve (see \cite{andreianov2021}), in a vicinity of the turning curve the spatial variation of $\rho$ may not be controlled; moreover, (ir)regularity of the turning curve itself impacts the modulus of continuity of $\rho$, making it an open question how to control time variations of $\rho$. We leave this issue for future research. \\
However, we can treat a natural modification of problem \eqref{eqElKhatibModel} for which the method applied for the affine cost \eqref{eq:affinecost} extends to general costs.   Let $\mathcal{R}: L^1((-\infty,T)) \longrightarrow L^1((0,T))$ be the operator defined by:
\begin{equation}\label{eq:mollifierPhi}
\mathcal{R}[\rho(\cdot,x)](t) :=
\delta \int_{-\infty}^t \rho(s,x) e^{-\delta (t-s)} \d s
\end{equation}
To make this operator well defined, we extend $\rho$ by $\rho(t) = \rho_0$ for any $t \in [-\infty,0]$. \textcolor{black}{This model corresponds to a memory effect in individual's perception of the density; $\mathcal{R}[\rho]$ is a \textit{subjective density} perceived by an agent making decision to move towards the most appropriate exit. Thus,} we consider the problem:
\begin{subequations}\label{eqMemoryElkhatib}
\begin{empheq}[left = \empheqlbrace ]{align}
\rho_t + \left[ \sign(x-\xi(t))\rho v(\rho)\right]_x  =  0 \tag{\ref{seqCLElKhatib}}\\
\int_{-1}^{\xi(t)} c(\mathcal{R}[\rho(\cdot,x)](t))dx  = \int_{\xi(t)}^1 c(\mathcal{R}[\rho(\cdot,x)](t))dx,   \tag{\ref{seqElKhatibEiko}'} \label{seqModifiedElKhatib}
\end{empheq}
\end{subequations}
with $c$ verifying \eqref{eqCostConditions}, and with initial datum satisfying \eqref{eq:HypRho_0}.

\smallskip
Equation \eqref{seqModifiedElKhatib} takes into account the average density over the recent past instead of the instantaneous density at a time $t$.
This models the bias, due to some inertia of human thinking, towards perception of the density for the pedestrians in the corridor; the quantity $\mathcal{R}[\rho(\cdot,x)]$ can be compared to other ``subjective densities'' used in the literature (cf. \cite{Carillo2016}, \cite{Andreianov2014CDA,Donadello2015}).
With the same calculations as \eqref{eqL1ContinuityCalculation}, we recover the term
$$I_2 = \int_{-1}^1 \Bigl| \mathcal{R}[\rho(\cdot,x)](t) - \mathcal{R}[\rho(\cdot,x)](s) \Bigr| \d x,$$
which is controlled by $2\delta \|\rho\|_{L^{\infty}} |t-s|$, a bound for the modulus of continuity of $\mathcal{R}[\rho(\cdot,x)]$.
For $I_1$ we can pass the absolute value inside the integral. Then $I_1$ is also controlled by the modulus of continuity of $\mathcal{R}[\rho(\cdot,x)]$.
Notice that we don't need the property \eqref{eqRhoSemiContinuity} for this reasoning. Consequently, we define:
\begin{equation}\label{eqBsetInertia}
  B_{2} = \left\lbrace \rho \in B_{L^1}(0,T \, \|\rho_0\|_{L^1}) \textrm{ s.t. } 0  \leq \rho \leq 1 \right\rbrace.
\end{equation}
Then,
$\mathcal{I}_\delta : (B_2, \|\cdot\|_{L^1((0,T)\times \mathbb{R})}) \longrightarrow (W^{1,\infty}((0,T)), \|\cdot\|_{\infty})$, $\rho\mapsto \xi$ where $\xi$ is defined by \eqref{seqModifiedElKhatib} with $\mathcal{R}$ given by \eqref{eq:mollifierPhi}, is well defined.
The analogue of Proposition \ref{thCoutLineaireLS} - where we use ${\mathcal I}_\delta$ instead of ${\mathcal I}_0$, we use $B_2$ instead of $B_1$ and we drop the assumption of affine cost - is easily justified. In particular,
the proof for the third item of this analogue of Proposition \ref{thCoutLineaireLS} holds with these choices. Thus, without the restriction \eqref{eq:affinecost} on the cost, we have the following claim:
\begin{proposition}
  Let $\rho_0 $ satisfy \eqref{eq:HypRho_0}. Let $c$ verifying \eqref{eqCostConditions}.
  Then problem \eqref{seqTrueCL}-\eqref{seqTrueCI}-\eqref{seqModifiedElKhatib} admits at least one solution.
\end{proposition}

\subsection{The general cost case with relaxed equilibrium}
We consider \eqref{eqConsideredModel} with a modified equilibrium equation \eqref{seqElKhatibEiko}. This time, we suppose that collective behavior of pedestrians makes appear some amount of inertia in the dynamics of $\xi$. Fixing $\epsilon > 0$, we consider as a simplest variant of such dynamics the ODE Cauchy problem
\begin{subequations}\label{eqInertiaEiko}
  \begin{empheq}[left = \empheqlbrace ]{align}
  -\epsilon \dot{\xi}(t) = \int_{\xi(t)}^1 c(\rho(t,x))dx - \int_{-1}^{\xi(t)} c(\rho(t,x))dx \\
  \int_{\xi(0)}^1 c(\rho_0(x))dx - \int_{-1}^{\xi(0)} c(\rho_0(x))dx = 0. \label{seqCIInertia}
\end{empheq}
\end{subequations}
for the $\rho$-driven evolution of the turning curve $\xi$.
Formally, the case $\epsilon=0^+$ corresponds to the standard Hughes's relation between the density and the turning curve; $\epsilon>0$ models a form of relaxation to the equilibrium given by this standard model. The primitive form of the Hughes' model, where the position of the turning curve is determined by an instantaneous Hamilton-Jacobi equation, should be modified to fit this dynamics of the turning curve; this modeling issue will be discussed elsewhere.
\begin{proposition}\label{thWellPosedInertia}
Let $\rho \in L^1((0,T)\times\mathbb{R})$. Let $c$ verifying the conditions \eqref{eqCostConditions}.
There exists a unique solution $\xi$ to the Cauchy problem \eqref{eqInertiaEiko}.
Furthermore, $\xi$ is Lipschitz and the Lipschitz constant is independent of $\rho$.
\end{proposition}

\begin{proof}
Let's denote:
$$\Psi(t,a) := \frac{1}{\epsilon} \left[ \int_{a}^1 c(\rho(t,x))dx - \int_{-1}^{a} c(\rho(t,x))dx \right].$$
Notice that for any $a,b \in [-1,1], t \in \mathbb{R}$,
\begin{align} \label{eqLipBoundInertia}
  \left| \Psi(t,a) - \Psi(t,b)\right|
  \leq \frac{1}{\epsilon} \left| \int_a^b 2c(\rho(t,x)) \d x \right|
  \leq \frac{2 \|c\|_{\infty}}{\epsilon}|a-b|.
\end{align}
We also have, for any $\xi$ such that $\|\xi\|_{\infty} \leq 1$:
\begin{align*}
  \left| \Psi(t,\xi(t)) \right|
  \leq \frac{1}{\epsilon}\left| \int_{-1}^{1} \sign(x-\xi(t)) c(\rho(t,x)) \d x \right|
  \leq \frac{ 2 \|c\|_{\infty}}{\epsilon}
\end{align*}
So $\Psi$ is Lipschitz with respect to the $a$ variable and uniformly bounded with respect to the $t$ variable.
We apply the Cauchy-Lipschitz Theorem and recover that there exists a unique local solution to the Cauchy problem \eqref{eqInertiaEiko}.
Using \eqref{eqLipBoundInertia}, we recover that the solution is global on $[0,T]$ and that $\xi$ is Lipschitz; moreover, the Lipschitz constant of $\xi$ does not depend on $\rho$.
\end{proof}

\begin{remark}
From Proposition \ref{thWellPosedInertia}, it follows that
$$\widetilde{\mathcal{I}_{\epsilon}} : L^1((0,T)\times \mathbb{R}, [0,1]) \longrightarrow W^{1,\infty}((0,T))$$
that maps any to $\rho$ to the unique $\xi$ solution to \eqref{eqInertiaEiko} is well defined.
\end{remark}

\begin{proposition}\label{thContinuityInertia}
Let $\rho_1, \rho_2 \in L^1((0,T)\times\mathbb{R})$. Let's denote $\xi_{1,2} := \widetilde{\mathcal{I}_{\epsilon}}(\rho_{1,2})$.
Then,
\begin{equation}\label{eqGronwallInertia}
  \|\xi_1 - \xi_2 \|_{\infty} \leq \frac{\|c'\|_{\infty}}{\epsilon}\exp\left[\frac{2T\|c\|_{\infty}}{\epsilon}\right] \, \|\rho_1 - \rho_2 \|_{L^1((0,T)\times (-1,1))}
\end{equation}
\end{proposition}

\begin{proof}
We denote $\xi_0$ the unique solution to \eqref{seqCIInertia}. Then, for any $t \in [0,T]$:
$$\xi_{1,2} = \xi_0 - \int_0^t \Psi_{1,2}(s, \xi_{1,2}(s)) \d s$$
Then, writing $\vee,\wedge$  for $\min,\max$, repsectively, we make the following calculations:
\begin{align*}
  &\xi_2(t) - \xi_1(t)\\
  &= \int_0^t \Psi_1(s, \xi_1(s)) - \Psi_2(s, \xi_2(s)) \d s \\
  &= \frac{1}{\epsilon} \int_0^t \left[ \int_{-1}^{\xi_1(s)} c(\rho_1(s,x)) \d x - \int_{\xi_1(s)}^{1} c(\rho_1(s,x)) \d x
  - \int_{-1}^{\xi_2(s)} c(\rho_2(s,x)) \d x + \int_{\xi_2(s)}^{1} c(\rho_2(s,x)) \d x\right] \d s \\
  &= \frac{1}{\epsilon} \int_0^t \biggl[ \int_{-1}^{(\xi_1 \vee \xi_2)(s)} c(\rho_1(s,x)) - c(\rho_2(s,x)) \d x \pm \int_{(\xi_1 \vee \xi_2)(s)}^{(\xi_1 \wedge \xi_2)(s)} c(\rho_1(s,x)) + c(\rho_2(s,x)) \d x\\
   & \qquad+ \int_{(\xi_1 \wedge \xi_2)(s)}^{1} c(\rho_2(s,x)) - c(\rho_1(s,x)) \d x\biggr] \d s
\end{align*}
And consequently,
\begin{align*}
  \left|\xi_1(t) - \xi_2(t)\right|
  &\leq \frac{1}{\epsilon} \int_0^t \int_{(\xi_1 \vee \xi_2)(s)}^{(\xi_1 \wedge \xi_2)(s)} c(\rho_1(s,x)) + c(\rho_2(s,x)) \d x \d s \\
  &+ \frac{1}{\epsilon} \int_0^t \int_{-1}^1 \left| c(\rho_1(s,x)) - c(\rho_2(s,x)) \right| \d s \d x \;=:\;J_1+J_2.
\end{align*}
For the term $J_2$ we can use the Lagrange inequality denoting $\|c'\|_{\infty} := \sup_{p\in [0,1]} |c'(p)|$. We get:
$$
  J_2 \leq \frac{\|c'\|_{\infty}}{\epsilon} \|\rho_1 - \rho_2 \|_{L^1((0,T)\times (-1,1))}.
$$
For the the term $J_1$,
 notice that, thanks to the cost conditions \eqref{eqCostConditions}, for any $s \in [0,t]$,
$$2 |\xi_1(s) - \xi_2(s)| \leq
\int_{(\xi_1 \vee \xi_2)(s)}^{(\xi_1 \wedge \xi_2)(s)} c(\rho_1(s,x))
+ c(\rho_2(s,x)) \d x \leq 2 \|c\|_{\infty} |\xi_1(s) - \xi_2(s)|$$
Consequently for any $s \in [0,T]$, there exists $\beta(s) \in [2 \, , \, 2 \, \|c\|_{\infty}]$ such that
$$\int_{(\xi_1 \vee \xi_2)(s)}^{(\xi_1 \wedge \xi_2)(s)} c(\rho_1(s,x))
+ c(\rho_2(s,x)) \d x = \beta(s)|\xi_1(s) - \xi_2(s)|.$$
Then $\beta \in L^{\infty}((0,T)) \subset L^1((0,T))$.
We are now in a position to use Gronwall's inequality with integrable coefficients. That inequality still holds without the continuity of $\beta$ if we use the Lebesgue differentiation Theorem. We thus reach to
\begin{align*}
  &\left|\xi_1(t) - \xi_2(t)\right|
  \leq  \int_0^t \frac{\beta(s)}{\epsilon} |\xi_1(s) - \xi_2(s)| \d s
  + \frac{\|c'\|_{\infty}}{\epsilon} \|\rho_1 - \rho_2 \|_{L^1}
 \end{align*}
which yields the subsequent estimates
\begin{align*}
  &\left|\xi_1(t) - \xi_2(t)\right| \leq \frac{\|c'\|_{\infty}}{\epsilon} \|\rho_1 - \rho_2 \|_{L^1} \exp \left[ \int_0^t \frac{\beta(s)}{\epsilon}\d s \right], \\
  & \|\xi_1 - \xi_2\|_{\infty} \leq \frac{\|c'\|_{\infty}}{\epsilon}\exp\left[\frac{2T\|c\|_{\infty}}{\epsilon} \right] \, \|\rho_1 - \rho_2 \|_{L^1}
\end{align*}
\end{proof}
\begin{remark}
One can check that, in the relaxed equilibrium setting, we never used any property of $\rho$ apart from the universal bounds $0\leq \rho \leq 1$. Consequently, in this case we also use:
\begin{equation}
  B_{2} = \left\lbrace \rho \in B_{L^1}(0,T \, \|\rho_0\|_{L^1}) \textrm{ s.t. } 0  \leq \rho \leq 1 \right\rbrace \tag{\ref{eqBsetInertia}}
\end{equation}
\end{remark}
Here's the final result in this relaxed equilibrium setting:
\begin{proposition}
  Let $\rho_0 $ satisfy \eqref{eq:HypRho_0}. Let $c$ verifying \eqref{eqCostConditions}. Then problem \eqref{seqTrueCL}-\eqref{seqTrueCI}-\eqref{eqInertiaEiko} admits at least one solution.
\end{proposition}
\begin{proof}
  We only have to apply Corollary \ref{thMainCorollary} with $B_{2}$ as a $B$ set and check that, using Propositions \ref{thWellPosedInertia} and \ref{thContinuityInertia}, all the assumptions on $\widetilde{\mathcal{I}_{\epsilon}}$ are satisfied.
\end{proof}

\section{Hughes' model with constrained evacuation at exit}\label{sec:constrained-exit}

In this section, we illustrate the robustness of our approach by modifying the Hughes model at the level of boundary conditions for the density, allowing for the realistic feature of capacity drop (see \cite{Andreianov2014CDA,Donadello2015} and references therein).
We consider the following dynamics for $\rho$ introduced in \cite{Andreianov2014CDA} on the basis of the theory of \cite{Colombo2007-ky,Andreianov2010-yh}:
\begin{subequations}\label{eqConstrainedCL}
  \begin{empheq}[left = \empheqlbrace]{align}
    \rho_t + &\left[ \sign(x - \xi(t)) f(\rho) \right]_x = 0 \\
    f(\rho(t, 1)) &\leq g \left( \int_\sigma^1 w_1(x) \rho(t,x) \d x \right) \label{seqExitR}\\
    f(\rho(t, -1)) &\leq g \left( \int_{-1}^{-\sigma} w_{-1}(x) \rho(t,x) \d x \right) \label{seqExitL}\\
    \rho(0,\cdot) &= \rho_0 (\cdot).
  \end{empheq}
\end{subequations}
The equations \eqref{seqExitR}-\eqref{seqExitL} prescribe the behaviour at exits situated at $x=\pm 1$; as in previous sections, we set up the conservation law for $\rho$ in the whole space, but the initial condition \eqref{eq:HypRho_0} is confined to the domain of interest $(-1,1)$. The flux $f(\rho)$ of pedestrian going through the exits is limited by respective constraints (we take a common nonlinearity $g$ for the sake of conciseness, but it is straightforward to extend the setting distinguishing $g_{1}$ and $g_{-1}$).
This flux limiter $g$ depends non locally of $\rho(t,\cdot)$ and of a weight $w$ supported in a vicinity of length $1- \sigma$ around the exits. This type of constraint models the well-known phenomenon of \textit{capacity drop} which, in extreme situations, corresponds to a panic behaviour at exits located at $x = \pm 1$, as discussed in \cite{Andreianov2014CDA} and \cite{Donadello2015}.
This model, allowing to consider constrained evacuation at exits, is phenomenologically more relevant than the model with open-end condition considered above (and it includes the previous model, for the trivial choice  $g\equiv \max_{[0,1]} f$, see Remark~\ref{rem:trivialconstraint}). As an example, this constrained evacuation model is able to reproduce the ``Faster is Slower'' effect at exits (see \cite{Donadello2015}).

\smallskip
In the following, we'll use the results of \cite{Donadello2015} and adapt them to our framework. We use the notations proposed in this paper:
\begin{itemize}
\item Since $f$ is concave positive such that $f(0)=f(1)$, there exists a $\bar{\rho} \in [0,1]$ such that $f'(\rho)(\bar{\rho} -\rho) > 0$ for a.e. $\rho \in [0,1]$.

\item We fix $\sigma \in (0,1)$. This is the threshold of influence on the exit, meaning that the pedestrian located before $x= \sigma$ have no influence on the exit congestion at $x=1$.
\end{itemize}

\smallskip
Let us take the strongest assumptions used in \cite{Andreianov2014CDA,Donadello2015}:
\begin{align}
\left\lbrace \begin{matrix}
 w_1 \in W^{1,\infty}((\sigma, 1],\mathbb{R}^+) \textrm{ s.t. } \int_{\sigma}^1 w_1 = 1 \\
  w_{-1} \in W^{1,\infty}([-1,-\sigma),\mathbb{R}^+) \textrm{ s.t. } \int_{-1}^{-\sigma} w_{-1} = 1
\end{matrix} \right.\label{eqConditionsW} \\
\label{eqConditionsG}
g \in W^{1,\infty}(\mathbb{R}^+,(0,f(\bar{\rho})]) \textrm{ is non-increasing.}
\end{align}

We can now introduce the notion of solution we'll use for $\rho$ combining the one in \cite{Colombo2007-ky} and Definition \ref{DefSolSansGerme}:
\begin{definition}\label{defSolconstrained}
Let $\xi \in W^{1,\infty}((0,T), (-1,1))$. Let $\rho_0 \in L^1(\mathbb{R},[0,1])$ supported in $[-1,1]$.
Let $f$  be a concave positive flux such that $f(0) = 0 = f(1)$ and $F(t,x,\rho) :=  \sign(x - \xi(t))f(\rho)$.
Let $g$, $\omega_{-1}$ and $\omega_1$ satisfy \eqref{eqConditionsW}-\eqref{eqConditionsG}. \\
We say that $\rho \in L^{1}((0,T)\times \mathbb{R})$ is an admissible solution to \eqref{eqConstrainedCL} if:\\

for all $\phi  \in \mathcal{C}^{\infty}_c((0,T) \times \mathbb{R})$,
  \begin{equation}\label{eqSolFaibleConst}
    \iint_{(0,T)\times \mathbb{R}} \rho \phi_t + F(t,x,\rho) \phi_x \d t \d x = 0,
  \end{equation}
moreover, setting
\begin{align}
  Q_{-1}(t) := g \left( \int_{-1}^{-\sigma} w_{-1}(x) \rho(t,x) \d x \right), \,\,
   Q_1(t) := g \left( \int_{\sigma}^{1} w_1(x) \rho(t,x) \d x \right),
\end{align}
there holds:

\smallskip
$\bullet$ For all positive $\phi \in \mathcal{C}^{\infty}-c(\{ x > \xi(t) \})$, for all $k \in \mathbb{R}$,
  \begin{align}\label{eqSolEntropyConstR}
     - \iint_{(0,T)\times \mathbb{R}} \left| \rho - k \right|\phi_t + q(\rho,k)\phi_x \d t \d x
     - 2 \int_{0}^T \left[ 1 - \frac{Q_1(t)}{f(\bar{\rho})}\right]f(k)  \phi(t,1) \d x 
     - \int_{\mathbb{R}} |\rho_0 - k| \phi(0,x) \d x \leq 0.
 \end{align}

\smallskip
$\bullet$ For all positive $\phi \in \mathcal{C}^{\infty}_c(\{ x < \xi(t) \})$,
for all $k \in \mathbb{R}$,
   \begin{align}\label{eqSolEntropyConstL}
      - \iint_{(0,T)\times \mathbb{R}} \left| \rho - k \right|\phi_t + q(\rho,k)\phi_x \d t \d x
      &- 2 \int_{0}^T \left[ 1 - \frac{Q_{-1}(t)}{f(\bar{\rho})}\right]\left(-f(k)\right) \phi(t,-1) \d x 
      - \int_{\mathbb{R}} |\rho_0 - k| \phi(0,x) \d x \leq 0 .
    \end{align}

   \smallskip
   $\bullet$ For all positive $\phi \in \mathcal{C}^{\infty}$ supported on $[a,b]$ such that $a<-1$, $1<b$ we have:
    	\begin{subequations}\label{eqFluxLimitp1}
      \begin{align}
   \int_0^T \int_a^{-1} \rho \phi_t + F(t,x,\rho) \phi_x \d t \d x \leq \int_0^T Q_{-1}(t) \phi(t,-1) \d t \\
       \int_0^T \int_1^b \rho \phi_t + F(t,x,\rho) \phi_x \d t \d x \leq \int_0^T Q_1(t) \phi(t,1) \d t
      \end{align}
    \end{subequations}

%

\end{definition}

\begin{remark}
 As detailled in \cite{Andreianov2010-yh}, equations \eqref{eqFluxLimitp1} 
 combined with the weak solution property \eqref{eqSolFaibleConst} imply that for a.e. $t \geq 0$, $f(\gamma_{L,R}^1 \rho(t)) \leq Q_1(t)$ and $-f(\gamma^{-1}_{L,R}\rho(t)) \geq -Q_{-1}(t)$. This corresponds to the expected limited flux condition.
\end{remark}

\begin{remark}\label{rem:trivialconstraint}
  One can notice that if for all $t \geq 0 , \, g(t) = f(\bar{\rho})$ then the flux is not limited at exits and $1 - \frac{Q_1(t)}{f(\bar{\rho})} = 1 - \frac{Q_{-1}(t)}{f(\bar{\rho})}  = 0$.
  Then, this definition is exactly Definition \ref{DefSolSansGerme}.
\end{remark}

\smallskip

We have the following results:
\begin{proposition}\label{thExistRho}
Let $\rho_0$ verify \eqref{eq:HypRho_0}. Let $\xi \in W^{1,\infty}((0,T),(-1,1))$.
There exists a solution to \eqref{eqConstrainedCL} in the sense of Definition \ref{defSolconstrained}.
\end{proposition}

The proof of Proposition \ref{thExistRho} is postponed to the Appendix. It is obtained via a convegent finite volume scheme. The details of the scheme and the proof of convergence can be found there.

Using the results from \cite{Colombo2007-ky}, \cite{Donadello2015}, \cite{Andreianov2014CDA} and a partitionning argument we prove a corollary of Theorem \ref{thContinuityOfS}:
\begin{corollary}\label{thUniquenessConstraint}
  Let $\rho_0$ verify \eqref{eq:HypRho_0}. Let $\xi \in W^{1,\infty}((0,T),(-1,1))$.
  There exists at most one solution $\rho$ of \eqref{eqConstrainedCL} in the sense of Definition \ref{defSolconstrained}.
Using Proposition \ref{thExistRho}, the solver operator
  $$\mathcal{S}_g : (W^{1,\infty}((0,T), (-1,1)), \|\cdot\|_{\infty}) \longrightarrow (L^1((0,T)\times (-1,1)), \|\cdot\|_{L^1}),$$
  that maps any $\xi$ to the unique solution $\rho$ to \eqref{eqConstrainedCL} is well defined and continuous.
\end{corollary}

\begin{proof}[Proof of Corollary \ref{thUniquenessConstraint}]
We use of the classical embedding of $W^{1,\infty}( [0,T], (-1,1))$ into $\mathcal{C}^0([0,T], (-1,1))$: there exists $K$ a closed segment of $(-1,1)$ such that $\xi \in \mathcal{C}^0([0,T],K)$.
We consider $(\phi_i)_{i \in \{-1,0,1 \} }$ a partition of the unity of an open set containing $[-1,1]$ such that:
\begin{align*}
  &\textrm{All the supports are segments and } 1 \in \supp(\phi_1), -1 \in \supp(\phi_{-1}) \textrm{ and } K \subset \supp(\phi_0) \subset (-1,1) \\
  &\left[\supp(\phi_{-1}) \cup \supp(\phi_1) \right]\bigcap K = \emptyset
\end{align*}
Let $\rho, \hat{\rho}$ be two solutions in the sense of Definition \ref{defSolconstrained}. We denote $\hat{Q}_{1,-1}$ the constraints associated with $\hat{\rho}$. Let $\Psi \in \mathcal{C}^{\infty}_c((0,T)\times \mathbb{R})$.
We use the classic Kruzkhov doubling of variables (cf. \cite{Evans1998-au}) in the open subdomains of $(0,T)\times\mathbb{R}$ situated between $x=-\infty$ and $x=-1$, $x=-1$ and $x=\xi(t)$, $x=\xi(t)$ and $x=1$, and finally between $x=1$ and $x=+\infty$.
Then by a limiting procedure analogous to the one employed in the proof of Theorem \ref{thPreuveUnicite}, we obtain the Kato inequality carrying singular terms concentrated on the three curves $\{ x = \xi(t) \}$, $\{ x= 1 \}$ and $\{ x = -1 \}$:
\begin{subequations}\label{eqKatoConstrained1}
\begin{align}
&- \iint_{(0,T)\times(-1,1)} |\rho-\hat{\rho}| \phi_t + q(\rho,\hat{\rho}) \phi_x \nonumber \\
& \leq \int_0^T \Psi(t,\xi(t)) \left(\phi_0 + \phi_{-1} + \phi_1 \right)(t,\xi(t))\left[ q^0_R(\gamma_R \rho, \gamma_R \hat{\rho} )
- q^0_L(\gamma_L \rho, \gamma_L \hat{\rho}) \right] \label{seqTermTC}\\
&+ \int_0^T \Psi(t,1) \phi_1(t,1) \left[ q^1(\gamma_R \rho, \gamma_R \hat{\rho} )
- q^1(\gamma_L \rho, \gamma_L \hat{\rho}) \right] \label{seqTermS1}\\
& + \int_0^T \Psi(t,-1) \phi_{-1}(t,-1) \left[ q^{-1}(\gamma_R \rho, \gamma_R \hat{\rho} )
- q^{-1}(\gamma_L \rho, \gamma_L \hat{\rho}) \right] , \label{seqTermSm1}
\end{align}
\end{subequations}
where the left and right traces are taken along their respective curves, and
\begin{align*}
&q^0_{L,R}(\rho, \hat{\rho}) := \sign(\rho - \hat{\rho})\left[
f_{L,R}(\rho) - f_{L,R}(\hat{\rho}) - \dot{\xi}(t) \left( \rho - \hat{\rho} \right)
\right] \\
&q^1(\rho, \hat{\rho}) := \sign(\rho - \hat{\rho})\left[
f_{R}(\rho) - f_{R}(\hat{\rho})
\right] \\
&q^{-1}(\rho, \hat{\rho}) := \sign(\rho - \hat{\rho})\left[
f_{L}(\rho) - f_{L}(\hat{\rho})
\right] .
\end{align*}

Referring to proof of Theorem \ref{thPreuveUnicite}, the integral \eqref{seqTermTC} is zero. Using the same argument as the proof of Proposition 2.10 in \cite{Andreianov2010-yh}, we get:
\begin{align*}
  \textrm{\eqref{seqTermS1}} \leq 2 \int_0^T \Psi(t,1) \left| Q_1(t) - \hat{Q}_1(t) \right| \d t \\
  \textrm{\eqref{seqTermSm1}} \leq 2 \int_0^T \Psi(t,-1) \left| Q_{-1}(t) - \hat{Q}_{-1}(t) \right| \d t
\end{align*}

As in the proof of Theorem \ref{thPreuveUnicite}, we integrate \eqref{eqKatoConstrained1} along a trapezoid $\mathcal{T}^{0,t}_{a,b}$. Then we use the definition of $Q_{\pm 1}$, $\hat Q_{\pm 1}$ with $L_g$ the Lipschitz constant of $g$ to get the following inequality:
\begin{equation*}
  \| \rho(t, \cdot) - \hat{\rho}(t,\cdot) \|_{L^1((a,b))} \leq
  \| \rho_0 - \hat{\rho}_0 \|_{L^1((a -L_f t,b +L_f t))} + 2 \int_0^t \int_{-1}^{1} L_g \left( \mathbb{1}_{(-1,-\sigma)} \omega_{-1} + \mathbb{1}_{(\sigma,1)} \omega_{1} \right) |\rho - \hat{\rho}| \d x \d s.
\end{equation*}
Eventually, using Holder's inequality and Gronwall's Lemma, we get:
\begin{equation}\label{eqModifiedTrapezoidProp}
  \| \rho(t, \cdot) - \hat{\rho}(t,\cdot) \|_{L^1((a,b))} \leq
  \| \rho_0 - \hat{\rho}_0 \|_{L^1((a -L_f t,b +L_f t))}
  e^{ C t },
\end{equation}
where $C := 2L_g\|\mathbb{1}_{(-1,-\sigma)} \omega_{-1} + \mathbb{1}_{(\sigma,1)} \omega_{1}\|_{\infty}$. Consequently, there is at most one solution in the sense of Definition \ref{defSolconstrained} associated to a fixed $\xi$ turning curve and an initial datum $\rho_0$. \\
In order to recover the continuity of the operator $S_g$ we proceed the same way as we proved Proposition \ref{thContinuityOfS}. We first cover any compact set contained in $\{ \xi(t)<x<1 \}$ by trapezoids.
Without loss of generality, we can suppose those trapezoids are at distance at least $\epsilon$ of the both interfaces $\{ x = \xi(t) \}$ and $\{ x =1 \}$.
Consequently, on any trapezoid, for all $n \geq n_0$, $\rho_n$ is a Kruzhkov entropy solution.
We recover compacity thanks to the averaging compactness lemma. This reasoning can be reproduced in the three other parts of the domain: $\{x < -1 \}, \{ -1 < x < \xi(t) \}$ and $\{ x > 1 \}$.
Then, we can pass to the limit via dominated convergence in equation \eqref{eqSolFaibleConst}
and in all the inequalities \eqref{eqSolEntropyConstR}-\eqref{eqSolEntropyConstL}-\eqref{eqFluxLimitp1}.
We conclude the proof with the same classical arguments as the proof of Proposition \ref{thContinuityOfS}.
That ends the proof of Corollary \ref{thUniquenessConstraint}.
\end{proof}
We are ready to state the main result of this section which is an analog of Theorem \ref{thMainCorollary}.
\begin{theorem}\label{thConstrain}
 Let $\rho_0$ verify \eqref{eq:HypRho_0}. Assume that $f$ verifies \eqref{NonDegen}.
  Let $g$ (resp. $\omega_{1,-1}$) satisfy \eqref{eqConditionsG} (resp. \eqref{eqConditionsW}). Let $B$ a convex closed bounded subset of $L^1((0,T) \times \mathbb{R})$
  and $$\mathcal{I}:  (B,\|\cdot\|_{L^1((0,T)\times \mathbb{R})}) \longrightarrow   (\mathcal{C}^0([0,T], \mathbb{R}), \|\cdot\|_{\infty})$$
  be a continuous operator such that $\forall \rho \in B, \, \forall t \in [0,T], \; \mathcal{I}[\rho](t) \in (-1,1)$.
  If there exists $r>0$ such that \eqref{seqLipStab}-\eqref{seqConditionBpossible} hold,
  then there exists $(\rho,\xi)$ a solution to the problem \eqref{eqConstrainedCL}-\eqref{seqTrueCI}-\eqref{seqTrueOp}. Here $\rho$ is a solution in the sense of Definition \ref{defSolconstrained}.
  In particular, existence is verified for $\mathcal{I} = \mathcal{I}_0$ (for affine cost) or with  $\mathcal{I}= \mathcal{I}_{\delta}$ or $ \widetilde{\mathcal{I}}_{\epsilon}$ (for general cost verifying \eqref{eqCostConditions}).
\end{theorem}

\appendix
\section{Convergence of the finite volume scheme in the constrained case}

In order to prove existence of a solution to \eqref{eqConstrainedCL} in the sense of Definition \ref{defSolconstrained}, we construct a converging finite volume scheme adapted around the fixed turning curve $\xi$. At the exits we use an operator splitting method with a scheme for the constraints $Q_1$ and $Q_{-1}$ as in \cite{Donadello2015}.\\
We now present the scheme used in this setting. Let $T, J \in \mathbb{N}$ such that:
\begin{equation}\label{numCFL}
2 \bigl( \|f'\|_{\infty} + \|\dot{\xi}\|_{\infty} \bigr) \frac{J}{T} \leq 1. \tag{CFL}
\end{equation}
We construct the following scheme:
\begin{subequations}
  \begin{align}
    &\Delta t = \frac{1}{T},
    &t^n := n \Delta t, \\
    &\Delta x = \frac{1}{J},
    &x_j = j \Delta x,\\
    &s^n := \frac{1}{\Delta t} \int_{t^n}^{t^{n+1}} \dot{\xi}(s) \d s ,
    &s_{\Delta}(t) := \sum_{1}^{N} \mathbb{1}_{[t^n, t^{n+1})}(t) s^n,\label{numSlope} \\
    &\xi_{\Delta} (t) := \xi(0) + \int_0^{t} s_{\Delta} (s) \d s ,
    &\xi^n = \xi_{\Delta}(t^n). \label{numXi}
    \end{align}
The discretization \eqref{numSlope}-\eqref{numXi} of the $\xi$ interface is detailled in \cite{Abraham2022} Section 3.1 where it is required to construct the adapted mesh. For any $n$, we denote $j_n$ the unique element of $\llbracket -J, J \rrbracket$ such that $\xi^n \in [x_{j_n}, x_{j_n + 1})$. We construct the following mesh:
  \begin{align}
    &\chi^n_j := \left\lbrace \begin{matrix}
    x_j \textrm{ if } j \leq j_n - 1 \\
    y^n \textrm{ if } j = j_n \\
    x_j \textrm{ if } j \geq j_n + 1 \\
  \end{matrix} \right.
  &\mathcal{P}_{j+1/2}^n := \left\lbrace \begin{matrix}
  (\chi_{j}^n ,\chi_{j+1}^n)\times (t^n, t^{n+1}) &\textrm{ if } j \leq j_n - 2 \\
  \\
  \textrm{the trapezoid } \chi^n_{j_n -1} \, \chi_{j_n -1}^{n+1} \, \chi^{n+1}_{j_{n+1}}\, \chi^n_{j_n} &\textrm{ if } j = j_n -1 \\
  \\
  \textrm{the trapezoid } \chi^n_{j_n} \, \chi_{j_{n +1}}^{n+1} \, \chi^{n+1}_{j_n +2} \, \chi^n_{j_n+2} &\textrm{ if } j = j_n  \\
  \\
  (\chi_{j+1}^n ,\chi_{j+2}^n)\times (t^n, t^{n+1}) &\textrm{ if } j \geq j_n + 1 \\
\end{matrix} \right.
  \end{align}
Notice that, thanks to the \eqref{numCFL} condition, $x_{j_n -1} < \xi^{n+1} < x_{j_n+2}$ so the trapezoids defined above
are never reduced to a triangle.
We denote $\underline{\mathcal{P}_{j+1/2}^n}$ (resp. $\overline{\mathcal{P}_{j+1/2}^n}$) the bottom (resp. top) segment of the tapezoid $\mathcal{P}_{j+1/2}^n$.
However, now that the mesh is modified we have two different partitions for the line $t=t^{n+1}$: $(\underline{\mathcal{P}^{n+1}_{j+1/2}})_{j \in \mathbb{Z}}$ and $(\overline{\mathcal{P}^{n}_{j+1/2}})_{j \in \mathbb{Z}}$.
We define $(\bar{\rho}^{n+1}_{i+1/2})_{i \in \mathbb{Z}}$ corresponding to the values of $\rho^{n+1}$ on $(\overline{\mathcal{P}^{n}_{i+1/2}})_{i \in \mathbb{Z}}$
and $({\rho}^n_{j+1/2})_{j \in \mathbb{Z}}$ the projection of this values on $(\underline{\mathcal{P}^n_{j+1/2}})_{j \in \mathbb{Z}}$.
  \begin{align}
    \bar{\rho}^{n+1}_{j + 1/2} = \frac{ \rho^n_{j+1/2} \biggl|\underline{\mathcal{P}_{j+1/2}^{n}}\biggr| - \Delta t (f^n_{j+1}- f^n_j) }{\biggl|\overline{\mathcal{P}_{j+1/2}^{n}}\biggr|}\\
    \rho^{n+1}_{j+1/2} := \frac{1}{\biggl|\underline{\mathcal{P}^{n+1}_{j+1/2}}\biggr|}\sum_{i \in \mathbb{Z}} \biggl| \underline{\mathcal{P}^{n+1}_{j+1/2}} \bigcap \overline{\mathcal{P}^{n}_{i+1/2}} \biggr| \, \bar{\rho}^{n+1}_{i+1/2} \\
    \rho_{\Delta}(t,x) := \sum_{n = 0}^N
    \sum_{\begin{matrix}
    j \in \mathbb{Z} \\
    j \neq j_n \pm 1
  \end{matrix}}
  \rho_{j+1/2}^n \, \mathbb{1}_{\mathcal{P}_{j+1/2}^n}(t,x)
  \end{align}
We now want to define the numerical fluxes $(f^n_j)_{j \in \mathbb{Z}}$ corresponding to the left and right edges of the trapezoids. It is worth noticing that we skipped $f^n_{j_n + 1}$ when we constructed the mesh.
We first define the non-local constraint approximation.
\begin{align}
  &\rho^n_{\Delta x}(\cdot) = \sum_{j \in \mathbb{Z}} \rho^n_{j+1/2} \mathbb{1}_{[\chi^n_j, \chi^n_{j+1})}(\cdot)\\
  &q_1^n := g_1\left(\int_{\sigma}^1 \rho_{\Delta x}^{n}(x) \omega_1(x)\d x \right) \\
  &q_{-1}^n := g_{-1}\left(\int_{-1}^{-\sigma} \rho_{\Delta x}^{n}(x) \omega_{-1}(x)\d x \right)\\
&F(\rho^n_{j-1/2},\rho^n_{j + 1/2}) = \left\{ \begin{matrix}
\min \left\{ \God_{f}(\rho^n_{j-1/2},\rho^n_{j + 1/2}) \, , \, q^n_1 \right\} && \textrm{ if } j - 1 = J \\
\max \left\{ \God_{-f}(\rho^n_{j-1/2},\rho^n_{j + 1/2}) \, , \, -q^n_{-1} \right\} && \textrm{ if } j = -J \\
\bold{F}_{int}^n(\rho^n_{j - 1/2},\rho^n_{j + 1/2}) && \textrm{ if } j=j_n \\
\God_{f}(\rho^n_{j-1/2},\rho^n_{j + 1/2}) && \textrm{ if } j > j_n \textrm{ and } j - 1 \neq J \\
\God_{-f}(\rho^n_{j-1/2},\rho^n_{j + 1/2}) && \textrm{ if } j < j_n \textrm{ and } j \neq -J.  \\
\end{matrix}\right.
\end{align}
Eventually, we define $\bold{F}_{int}^n$ as in \cite{AbrahamPreprint} (see details in Subsections 2.5, 3.3 and 5.1):
\begin{align}
  &f^n_{L,R}(\rho) := \pm f(\rho) - s^n \rho \notag \\
  &\forall (\rho_L,\rho_R) \in [0,1]^2, \, \exists k \in [0,1] \textrm{ s.t. }
  \bold{God}_{f^n_L}(\rho_L,k) = \bold{God}_{f^n_R}(k, \rho_R) \notag \\
  &\bold{F}_{int}^n(\rho^n_{j - 1/2},\rho^n_{j + 1/2}) := \bold{God}_{f^n_L}(\rho^n_{j - 1/2},k) = \bold{God}_{f^n_R}(k, \rho^n_{j + 1/2})
\end{align}
\end{subequations}
Numerical simulations with for this scheme can be found in \cite[Sect. 5.1]{AbrahamPreprint} for the case of open-end condition at exits.

We are now in a position to start the proof of convergence, which merely assembles with the help of the partition-of-unity technique of \cite{Abraham2022,AbrahamPreprint} the arguments from \cite{AbrahamPreprint} (for the inner interface situated at $x=\xi(t)$ and \cite{Donadello2015} (for the constraints set at $x=\pm 1$).
\begin{proof}[Proof of Proposition \ref{thExistRho}]
 The proof follows the general idea of \cite[Sect. 4]{Abraham2022}, see also \cite{AbrahamPreprint}.
Since the interfaces $\{x=-1 \}$, $\{ x = \xi(t)\}$ and $\{ x = 1 \}$ are non-intersecting, we isolate them in the supports of a partition of unity $\phi_{-1}$, $\phi_0$ and $\phi_1$. We fix a test function $\phi$.
Taking (the discretization of) the test function $\phi_0\phi$ we can use the specific result for the Hughes' model treated in \cite[Sect. 5.1]{AbrahamPreprint} to recover the approximate entropy inequalities satisfied by the discrete solution, with the test function $\phi_0\phi$.
For test functions $\phi_{-1}\phi$ and $\phi_1\phi$, we use in the same way the result of \cite[Prop.\,3.1]{Donadello2015}. Summing up the contributions of the three parts of the partition of unity, we obtain approximate entropy inequality for the discrete solution, with arbitrary test function $\phi$. In addition, the integral weak formulation for the approximate solution follows from the scheme's conservativity. We use the same compactness argument as in \cite[Sect.\,3.4]{Abraham2022}.
 We can pass to the limit in the approximate weak formulation and in the approximate entropy inequalities, for the chosen converging subsequence and arbitrary test function. This allows us to characterize the limit as an entropy solution in the sense of Definition \ref{defSolconstrained} of the problem at hand. Finally, thanks to the uniqueness proven in Theorem \ref{thUniquenessConstraint}, the whole sequence of discrete solutions converges to the unique solution in the sense of Definition \ref{defSolconstrained}.
\end{proof}

\section*{Acknowledgments}
This paper has been supported by the RUDN University Strategic Academic Leadership Program.

\bibliographystyle{siamplain}
\bibliography{citation}
\end{document}